\documentclass{amsart}

\usepackage{enumerate, amsmath, amsfonts, amssymb, amsthm, wasysym, graphics, graphicx, xcolor, url, hyperref, hypcap, a4wide, pdflscape, multido, xargs, colortbl,, multicol, multirow, overpic}
\hypersetup{colorlinks=true, citecolor=darkblue, linkcolor=darkblue}
\usepackage[all]{xy}
\usepackage{tikz}\usetikzlibrary{trees,snakes,shapes,arrows,matrix,calc}
\graphicspath{{figures/}}

\title{cluster algebras of type~$D$: pseudotriangulations~approach}

\author[C.~Ceballos]{Cesar Ceballos$^{\star}$} 
\address[C.~Ceballos]{Department of Mathematics and Statistics, York University, Toronto}
\email{ceballos@mathstat.yorku.ca}
\urladdr{http://garsia.math.yorku.ca/~ceballos/}
\thanks{$^\star$CC was supported by the government of Canada through a Banting Postdoctoral Fellowship. He was also supported by a York University research grant.}

\author[V.~Pilaud]{Vincent Pilaud$^{\ddagger}$} 
\address[V.~Pilaud]{CNRS \& LIX, \'Ecole Polytechnique, Palaiseau}
\email{vincent.pilaud@lix.polytechnique.fr}
\urladdr{http://www.lix.polytechnique.fr/~pilaud/}
\thanks{$^\ddagger$VP was partially supported by the Spanish MICINN grant MTM2011-22792 and the French ANR grant EGOS (12 JS02 002 01).}

\newtheorem{theorem}{Theorem}

\newtheorem{proposition}[theorem]{Proposition}

\theoremstyle{definition}
\newtheorem{example}[theorem]{Example}
\newtheorem{remark}[theorem]{Remark}

\newcommand{\R}{\mathbb{R}} 
\newcommand{\cA}{\mathcal A} 
\newcommand{\cM}{\mathcal M} 
\newcommand{\configD}{\mathbf D} 
\renewcommand{\b}[1]{\mathbf{#1}} 
\newcommand{\disk}{D} 

\newcommand{\set}[2]{\left\{ #1 \;\middle|\; #2 \right\}} 
\newcommand{\ssm}{\smallsetminus} 
\newcommand{\eqdef}{\mbox{\,\raisebox{0.2ex}{\scriptsize\ensuremath{\mathrm:}}\ensuremath{=}\,}} 
\newcommand{\diag}[2]{[#1,#2]} 
\newcommand{\diagD}[2]{\ifthenelse{\equal{#2}{L}}{#1^{\textsc l}}{#1^{\textsc r}}} 
\newcommand{\pseudotriangle}{\raisebox{-1pt}{\includegraphics{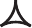}}} 
\newcommand{\pseudoquadrangle}{\raisebox{-1pt}{\includegraphics{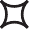}}} 
\newcommand{\Pseudo}{\mathsf{Pseudo}} 
\newcommand{\diagToVar}{\chi} 
\newcommandx{\cross}[2][1=\theta, 2=\delta]{[\,#1 \, \| \, #2 \,]} 
\newcommand{\quiver}{Q} 
\newcommand{\posToDiag}{\zeta} 
\newcommandx{\posToDiagc}[1][1=]{\posToDiag_{\sqc}^{#1}} 
\newcommand{\sqc}{{\mathrm c}} 
\newcommand{\Q}{{\mathrm Q}} 
\newcommand{\Qc}{\Q_\sqc} 
\newcommand{\wo}{w_\circ} 
\newcommand{\woc}{{\mathrm w}_\circ(\sqc)} 
\newcommand{\cwoc}{{\sqc\woc}} 
\newcommand{\modulo}[1]{\, (\text{mod } #1)} 
\newcommand{\tausqc}{\tau_\sqc} 
\newcommand{\subwordComplex}{\mathcal{SC}} 
\DeclareMathOperator{\conv}{conv} 

\newcommand{\fref}[1]{Figure~\ref{#1}} 
\newcommand{\ie}{\textit{i.e.}~} 
\newcommand{\eg}{\textit{e.g.}~} 
\newcommand{\para}[1]{\medskip\noindent\textbf{#1~---}} 
\definecolor{darkblue}{rgb}{0,0,0.7} 
\newcommand{\darkblue}{\color{darkblue}} 
\newcommand{\defn}[1]{\emph{\darkblue #1}} 
\usepackage{todonotes}

\definecolor{Gray}{gray}{0.9} 
\newcolumntype{g}{>{\columncolor{Gray}}c}
\newcolumntype{x}{@{\hspace{.1cm}}c@{\hspace{.1cm}}}
\newcommand{\twolines}[3]{\begin{minipage}{#1} \begin{center} \vspace*{.1cm} #2 \\ #3 \vspace*{.1cm} \end{center} \end{minipage}} 

\begin{document}

\vspace*{.7cm}

\begin{abstract}
We present a combinatorial model for cluster algebras of type~$D_n$ in terms of centrally symmetric pseudotriangulations of a regular~$2n$-gon with a small disk in the centre. This model provides convenient and uniform interpretations for clusters, cluster variables and their exchange relations, as well as for quivers and their mutations. We also present a new combinatorial interpretation of cluster variables in terms of perfect matchings of a graph after deleting two of its vertices. This interpretation differs from known interpretations in the literature. Its main feature, in contrast with other interpretations, is that for a fixed initial cluster seed, one or two graphs serve for the computation of all cluster variables. Finally, we discuss applications of our model to polytopal realizations of type~$D$ associahedra and connections to subword complexes and $c$-cluster complexes.
\end{abstract}

\maketitle

\vspace{-.4cm}


\defn{Cluster algebras}, introduced by S.~Fomin and A.~Zelevinsky in~\cite{FominZelevinsky-ClusterAlgebrasI, FominZelevinsky-ClusterAlgebrasII}, are commutative rings generated by a set of \defn{cluster variables}, which are grouped into overlapping \defn{clusters}. The clusters can be obtained from any \defn{initial cluster seed}~$X = \{x_1, \dots, x_n\}$ by a mutation process. Each mutation exchanges a single variable~$y$ to a new variable~$y'$ satisfying a relation of the form~${yy' = M_+ + M_-}$, where~$M_+$ and~$M_-$ are monomials in the variables involved in the current cluster and distinct from~$y$ and~$y'$. The precise content of these monomials~$M_+$ and~$M_-$ is controlled by a combinatorial object (a skew-symmetrizable matrix, or equivalently a weighted quiver~\cite{Keller}) which is attached to each cluster and is also transformed during the mutation. We refer to~\cite{FominZelevinsky-ClusterAlgebrasI} for the precise definition of these joint dynamics.

In~\cite[Theorem~3.1]{FominZelevinsky-ClusterAlgebrasI}, S.~Fomin and A.~Zelevinsky proved the Laurent phenomenon for cluster algebras: given any initial cluster seed $X = \{x_1, \dots, x_n\}$, all cluster variables obtained during the mutation process are \defn{Laurent polynomials} (quotients of polynomials by monomials) in the variables~$x_1, \dots, x_n$. Note that we think of the cluster variables as a set of variables satisfying some algebraic relations. These variables can be expressed in terms of the variables in any initial cluster seed~$X = \{x_1,\dots,x_n\}$ of the cluster algebra. Starting from a different cluster seed~$X' = \{x'_1, \dots, x'_n\}$ would give rise to an isomorphic cluster algebra, expressed in terms of the variables~$x'_1, \dots, x'_n$ of this seed. For concrete computations on specific examples, it is often important to obtain the expression of arbitrary variables in terms of arbitrary initial cluster seeds, and preferable to avoid the time consuming mutation process when possible.

Finite type cluster algebras, \ie cluster algebras whose mutation graph is finite, were classified in~\cite[Theorem~1.4]{FominZelevinsky-ClusterAlgebrasII} using the Cartan-Killing classification for finite crystallographic root systems. Finite type cluster algebras motivated research on Coxeter-Catalan combinatorics~\cite{FominZelevinsky-YSystems, Reading-coxeterSortable, Athanasiadis1, Athanasiadis2, Armstrong} and on constructions of generalized associahedra~\cite{ChapotonFominZelevinsky, HohlwegLangeThomas, Stella, PilaudStump-brickPolytope}.

In~\cite[Section~3.5]{FominZelevinsky-YSystems}\cite[Section~12.4]{FominZelevinsky-ClusterAlgebrasII}, S.~Fomin and A.~Zelevinsky introduced geometric models to concretely manipulate cluster algebras of types~$A$, $B$, $C$, and~$D$. For example, in the cluster algebra of type~$A_n$, cluster variables correspond to internal diagonals of an $(n+3)$-gon and clusters correspond to its triangulations. The quiver of a cluster~$X$ is the diagonal-rotation graph of the triangulation~$X$ (its nodes are the diagonals of~$X$ and its arcs connect diagonals which are clockwise consecutive in a triangle of~$X$). Cluster mutations correspond to flips between triangulations, and the exchange relation is given by an analog of Ptolemy's relation between the length of the diagonals and of the edges of a quadrilateral~$pqrs$ inscribed on a circle: ${|pr| \cdot |qs| = |pq| \cdot |rs| + |ps| \cdot |qr|}$. These models provide useful tools to get intuition and to experiment on combinatorial properties of the corresponding cluster algebras. Moreover, many properties of the cluster algebra can be read directly from the geometric model. For example, for a cluster variable~$y$ and an initial cluster seed~$X$ in type~$A$, the denominator of~$y$ with respect to~$X$ is the product of the variables in~$X$ whose diagonals cross the diagonal~$y$, while the numerator of~$y$ can be computed in terms of perfect matching enumeration in a well-chosen weighted bipartite graph~\cite{CarrollPrice, Prop}. More precisely, consider the vertex-triangle incidence graph~$G$ of the triangulation~$X$ where the edge joining a vertex~$v$ to a triangle~$t$ is weighted by the variable of the diagonal of~$t$ opposite to~$v$, and where the two endpoints of the diagonal~$y$ are deleted. The variable~$y$ is obtained as the sum of the weights of all perfect matchings of~$G$ divided by the product of the variables in~$X$ (where the weight of a perfect matching is the product of the weights of its edges). \fref{fig:matchingExample0} illustrates this computation. These perfect matching enumeration schemes were extended to certain families of infinite cluster algebras arising from triangulated surfaces in~\cite{SchifflerThomas, Schiffler, MusikerSchiffler, MusikerSchifflerWilliams}.

\begin{figure}
	\centerline{\begin{overpic}[width=.85\textwidth]{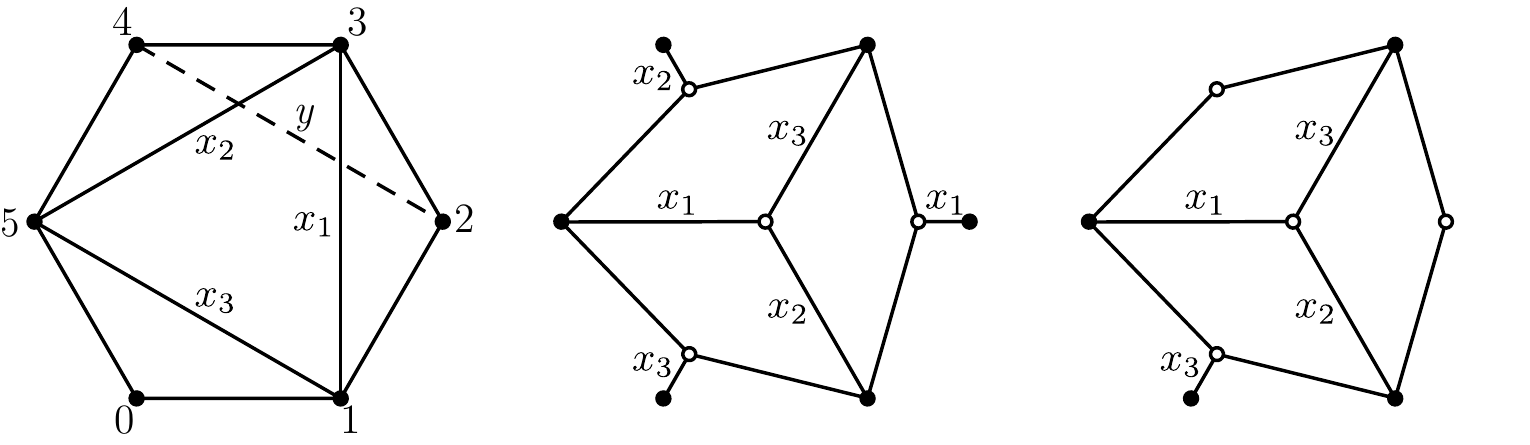} \put(15,-2){$T$} \put(85,-2){$G$} \end{overpic} \raisebox{1.7cm}{$\displaystyle y = \frac{x_3(x_1+x_2+x_3)}{x_1x_2x_3}$}}
	\caption{G.~Carroll and G.~Price's computation of a cluster variable in terms of perfect matchings in type~$A$~\cite{CarrollPrice, Prop}.}
	\label{fig:matchingExample0}
\end{figure}

In this paper, we focus on the cluster algebras of type~$D$ with both a combinatorial and an algebraic perspective. Different models exist for these cluster algebras, in terms of centrally symmetric triangulations of polygons with bicolored long diagonals~\cite[Section~3.5]{FominZelevinsky-YSystems}\cite[Section~12.4]{FominZelevinsky-ClusterAlgebrasII}, or in terms of tagged triangulations of a punctured $n$-gon~\cite{FominShapiroThurston}. In this paper, we propose an alternative geometric model for the cluster algebra of type~$D_n$ based on pseudotriangulations of a $2n$-gon with a small disk in the center. The six sections of this paper present the following concrete combinatorial, algebraic, and geometric applications of this model:
\begin{enumerate}[1.]
\item Cluster mutations are interpreted as flips in pseudotriangulations. Moreover, all exchange relations are described by a uniform algebraic relation between the variables on the diagonals and on the boundary of the pseudoquadrangles involved in the flips.
\item The quiver of a cluster~$X$ is the chord-rotation graph of the pseudotriangulation~$X$ (its nodes are the chords of~$X$ and its arcs connect all chords located on clockwise consecutive sides of pseudotriangles of~$X$).
\item The expression of any cluster variable~$y$ with respect to a fixed cluster seed~$X$ can be computed in terms of perfect matching enumeration of a weighted bipartite graph after deletion of two vertices determined by~$y$. The advantage of our model is that once the cluster seed~$X$ is fixed, one or two graphs serve for all cluster variables. These graphs are essentially the vertex-pseudotriangle incidence graph of the pseudotriangulation~$X$.
\item The model provides alternative polytopal realizations of the type~$D$ associahedron based on the polytope of pseudotriangulations constructed by G.~Rote, F.~Santos, and I.~Streinu~\cite{RoteSantosStreinu-pseudotriangulationPolytope}.
\item The pseudotriangulation approach also provides a geometric interpretation of the description of the type~$D$ $c$-cluster complex as a subword complex~\cite{CeballosLabbeStump}. This interpretation is closely related to the duality between pseudotriangulations and pseudoline arrangements presented by V.~Pilaud and M.~Pocchiola in~\cite{PilaudPocchiola}.
\item The $c$-cluster complex can be described in a purely combinatorial way in terms of pseudotriangulations. 
\end{enumerate}

We observe that the validity of our model can be argued by three distinct methods: either comparing our model with the classical model of S.~Fomin and A.~Zelevinsky~\cite[Section~3.5]{FominZelevinsky-YSystems}\cite[Section~12.4]{FominZelevinsky-ClusterAlgebrasII} (see Remark~\ref{rem:comparisonOldModel}), or comparing flips in pseudotriangulations with mutations in quivers (see Proposition~\ref{prop:quiverMutation}), or connecting pseudotriangulations to the clusters of type~$D$ via subword complexes (see Section~\ref{sec:subwordComplex}). We believe that it is not necessary to develop these arguments in detail. In fact, this paper should rather be considered as a tool to make by hand examples and explicit computations in type~$D$. Algebraic computations in this model are as elementary and uniform as in the triangulation model for type~$A$, and some combinatorial properties even become much simpler than their type~$A$ analogues (\eg the diameter of the type~$D$ associahedron, see~\cite{CeballosPilaud-diameterTypeD}).
Other works in type $D$ cluster algebras include~\cite{baur_frieze_2008,gunawan_t-path_2014}.


\section{Pseudotriangulations model}
\label{sec:model}

In this section, we present a combinatorial model for the cluster algebra of type~$D_n$ in terms of pseudotriangulations of a geometric configuration~$\configD_n$. Even if this model is closely related to the geometric model of S.~Fomin and A.~Zelevinsky for type~$D_n$~\cite[Section~3.5]{FominZelevinsky-YSystems}\cite[Section~12.4]{FominZelevinsky-ClusterAlgebrasII}, we prefer to use pseudotriangulations as they simplify and make uniform  combinatorial interpretations of clusters, cluster variables and exchange relations. The connection between our model and the classical one is discussed in more details in Remark~\ref{rem:comparisonOldModel}.

\begin{figure}[b]
	\centerline{\includegraphics[scale=1]{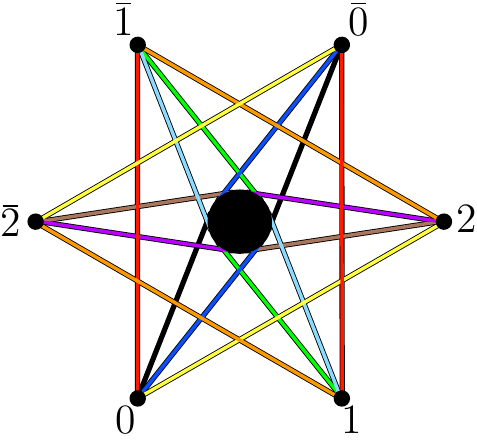} \quad \includegraphics[scale=1]{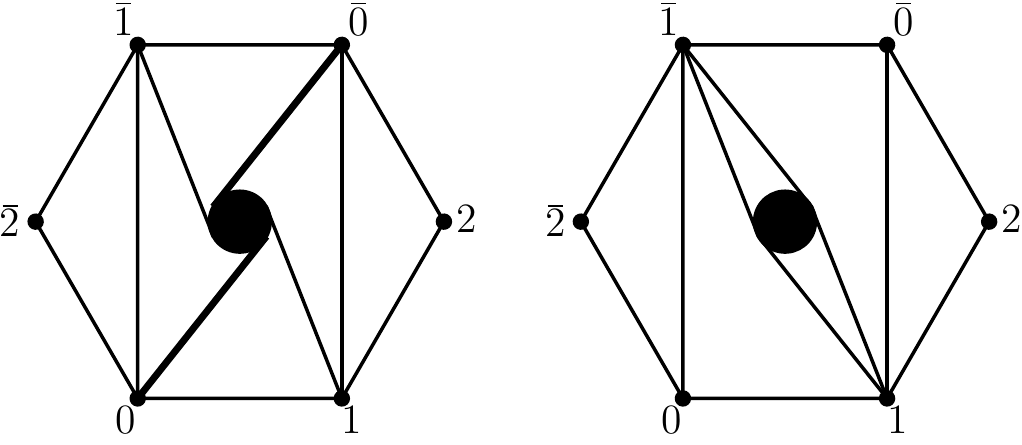}}
	\caption{The configuration~$\configD_3$ has $9$ centrally symmetric pairs of chords (left). A centrally symmetric pseudotriangulation~$T$ of~$\configD_3$ (middle). The centrally symmetric pseudotriangulation of~$\configD_3$ obtained from~$T$ by flipping the chords~$0^{\textsc r}$ and~$\bar 0^{\textsc r}$~(right).}
	\label{fig:configurationDn&Flips}
\end{figure}

We consider a regular convex~$2n$-gon, together with a disk~$\disk$ placed at the center of the $2n$-gon, whose radius is small enough such that~$\disk$ only intersects the long diagonals of the $2n$-gon. We denote by~$\configD_n$ the resulting configuration, see \fref{fig:configurationDn&Flips}. We call \defn{chords} of~$\configD_n$
\begin{itemize}
\item all the diagonals of the $2n$-gon, except the long ones, and
\item all the segments tangent to the disk~$\disk$ and with one endpoint among the vertices of the $2n$-gon. Note that each vertex~$p$ is incident to two such chords; we denote by~$\diagD{p}{L}$ (resp.~by~$\diagD{p}{R}$) the chord emanating from~$p$ and tangent on the left (resp.~right) to the disk~$\disk$. We call these chords \defn{central}.
\end{itemize}
Cluster variables, clusters, exchange relations, compatiblity degrees, and denominators of cluster variables in the cluster algebra~$\cA(D_n)$ can be interpreted geometrically as follows:
\begin{enumerate}[(i)]
\item Cluster variables correspond to \defn{centrally symmetric pairs of (internal) chords} of the geometric configuration~$\configD_n$. See \fref{fig:configurationDn&Flips}\,(left). To simplify notations, we identify a chord~$\delta$, its centrally symmetric copy~$\bar \delta$, and the pair~$\{\delta, \bar \delta\}$. We denote by~$\diagToVar_\delta = \diagToVar_{\bar \delta}$ the cluster variable corresponding to the pair of chords~$\{\delta, \bar \delta\}$.

\item Clusters correspond to \defn{centrally symmetric pseudotriangulations} of~$\configD_n$ (\ie maximal centrally symmetric crossing-free sets of chords of~$\configD_n$). Each pseudotriangulation of~$\configD_n$ contains exactly~$2n$ chords, and partitions $\conv(\configD_n) \ssm \disk$ into \defn{pseudotriangles}  (\ie interiors of simple closed curves with three convex corners related by three concave chains), see \fref{fig:configurationDn&Flips}. We refer to~\cite{RoteSantosStreinu-survey} for a detailed survey on pseudotriangulations, including their history, motivations, and applications.

\item Cluster mutations correspond to \defn{flips} of centrally symmetric pairs of chords between centrally symmetric pseudotriangulations of~$\configD_n$. A flip in a pseudotriangulation~$T$ replaces an internal chord~$e$ by the unique other internal chord~$f$ such that~$(T \ssm e) \cup f$ is again a pseudotriangulation of~$T$. To be more precise, deleting~$e$ in~$T$ merges the two pseudotriangles of~$T$ incident to~$e$ into a pseudoquadrangle~$\pseudoquadrangle$ (\ie the interior of a simple closed curve with four convex corners related by four concave chains), and adding~$f$ splits the pseudoquadrangle~$\pseudoquadrangle$ into two new pseudotriangles. The chords~$e$ and~$f$ are the two unique chords which lie both in the interior of~$\pseudoquadrangle$ and on a geodesic between two opposite corners of~$\pseudoquadrangle$. We refer again to~\cite{RoteSantosStreinu-survey} for more details. 

For example, the two pseudotriangulations of \fref{fig:configurationDn&Flips} are related by a centrally symmetric pair of flips. We have represented different kinds of flips between centrally symmetric pseudotriangulations of the configuration~$\configD_n$ in \fref{fig:typeDflip}. The flip graphs on centrally symmetric pseudotriangulations of~$\configD_3$ and~$\configD_4$ are illustrated in Figures~\ref{fig:typeD3associahedron} and~\ref{fig:typeD4associahedron}.

\begin{figure}
	\centerline{\includegraphics[width=\textwidth]{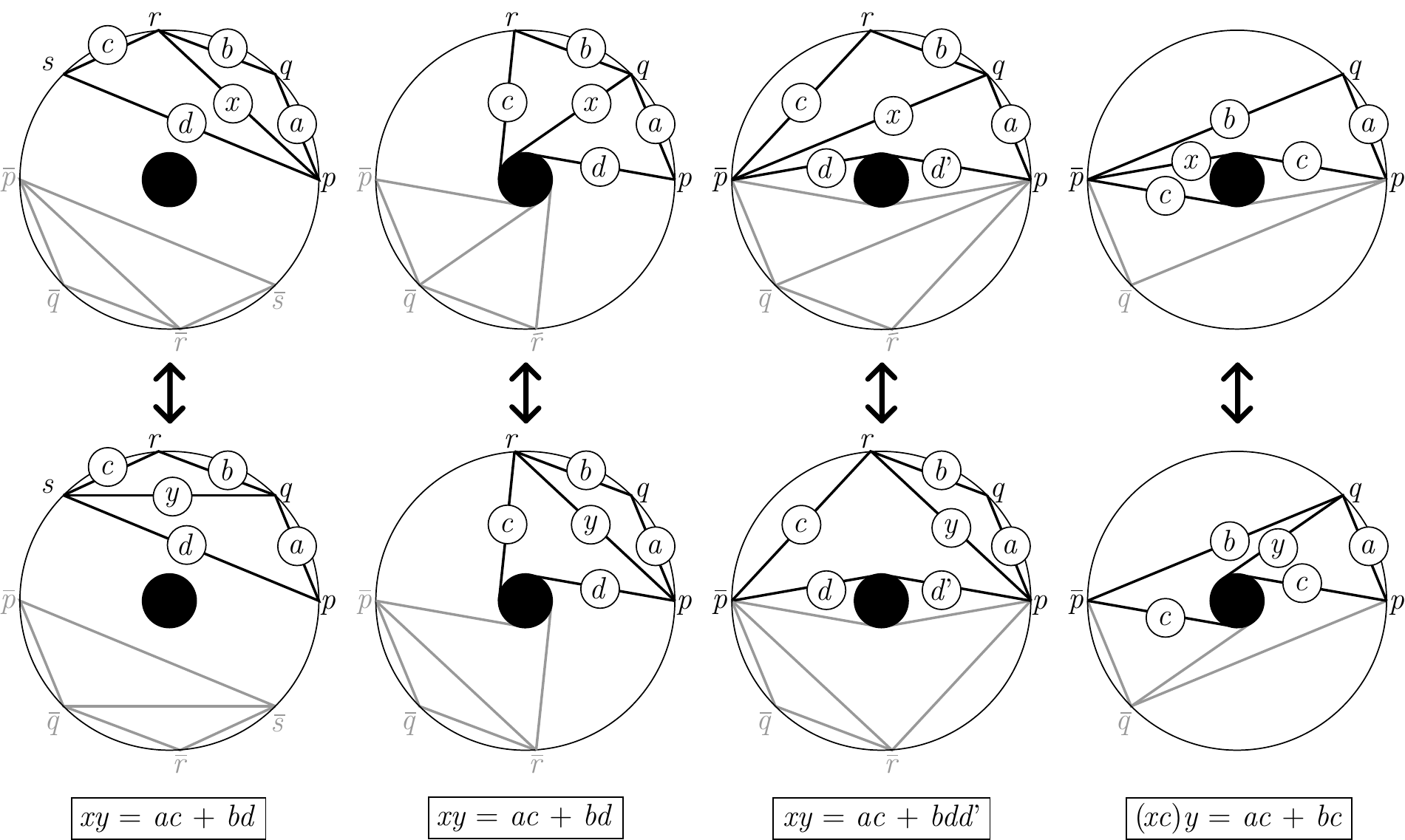}}
	\caption{Different kinds of flips and exchange relations in type~$D$.}
	\label{fig:typeDflip}
\end{figure}

\begin{figure}
        \centerline{\includegraphics[width=.55\textwidth]{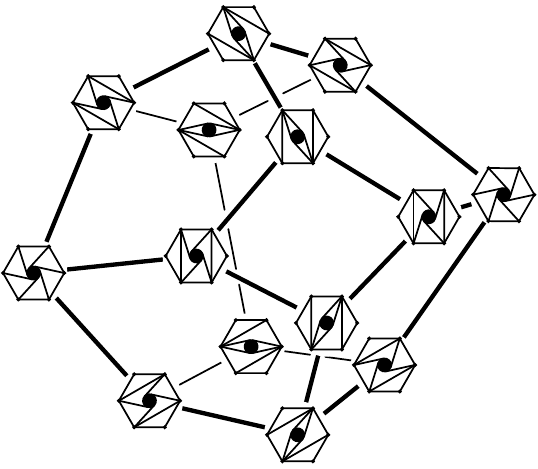}}
        \caption{The type~$D_3$ mutation graph interpreted geometrically by centrally symmetric pseudotriangulations of~$\configD_3$. Note that this graph is the $1$-skeleton of the $3$-dimensional associahedron since~$D_3 = A_3$.}
        \label{fig:typeD3associahedron}
\end{figure}

As in type~$A$, the exchange relations between cluster variables during a cluster mutation can be understood in the geometric picture. More precisely, flipping~$e$ to~$f$ in the pseudoquadrangle~$\pseudoquadrangle$ with convex corners~$\{p,q,r,s\}$ (and simultaneously $\bar e$ to~$\bar f$ in the centrally symmetric pseudoquadrangle~$\bar \pseudoquadrangle$) results in the exchange relation
\[
\hspace{1cm}
\Pi(\pseudoquadrangle,p,r) \cdot \Pi(\pseudoquadrangle,q,s) = \Pi(\pseudoquadrangle,p,q) \cdot \Pi(\pseudoquadrangle,r,s) + \Pi(\pseudoquadrangle,p,s) \cdot \Pi(\pseudoquadrangle,q,r),
\]
where
\begin{itemize}
\item $\Pi(\pseudoquadrangle,p,r)$ denotes the product of the cluster variables~$\diagToVar_\delta$ corresponding to all chords~$\delta$ which appear along the geodesic from~$p$ to~$r$ in~$\pseudoquadrangle$ --- and similarly for $\Pi(\pseudoquadrangle,q,s)$ --- and
\item $\Pi(\pseudoquadrangle,p,q)$ denotes the product of the cluster variables~$\diagToVar_\delta$ corresponding to all chords~$\delta$ which appear on the concave chain from~$p$ to~$q$ in~$\pseudoquadrangle$ --- and similarly for $\Pi(\pseudoquadrangle,q,r)$, $\Pi(\pseudoquadrangle,r,s)$, and~$\Pi(\pseudoquadrangle,p,s)$.
\end{itemize}
For example, the four flips in \fref{fig:typeDflip} result in the following relations:
\renewcommand{\arraystretch}{1.1}
\[
\hspace{1cm}
\begin{array}{@{}c@{\ }c@{\ }c@{\ }c@{\ }c@{\quad }l}
\diagToVar_{\diag{p}{r}} \cdot \diagToVar_{\diag{q}{s}} & = & \diagToVar_{\diag{p}{q}} \cdot \diagToVar_{\diag{r}{s}} & + & \diagToVar_{\diag{p}{s}} \cdot \diagToVar_{\diag{q}{r}}, \\[3pt]
\diagToVar_{\diag{p}{r}} \cdot \diagToVar_{\diagD{q}{R}} & = & \diagToVar_{\diag{p}{q}} \cdot \diagToVar_{\diagD{r}{R}} & + & \diagToVar_{\diagD{p}{R}} \cdot \diagToVar_{\diag{q}{r}}, \\[3pt]
\diagToVar_{\diag{p}{r}} \cdot \diagToVar_{\diag{q}{\bar p}} & = & \diagToVar_{\diag{p}{q}} \cdot \diagToVar_{\diag{r}{\bar p}} & + & \diagToVar_{\diagD{\bar p}{L}} \cdot \diagToVar_{\diagD{p}{R}} \cdot \diagToVar_{\diag{q}{r}}, \\[3pt]
\diagToVar_{\diagD{\bar p}{L}} \cdot \diagToVar_{\diagD{p}{R}} \cdot \diagToVar_{\diagD{q}{R}} & = & \diagToVar_{\diag{p}{q}} \cdot \diagToVar_{\diagD{\bar p}{R}} & + & \diagToVar_{\diag{q}{\bar p}} \cdot \diagToVar_{\diagD{p}{R}}.
\end{array}
\]
Note that the last relation will always simplify by~$\diagToVar_{\diagD{p}{R}} = \diagToVar_{\diagD{\bar p}{R}}$. For a concrete example, in the flip presented in \fref{fig:configurationDn&Flips}, we obtain the relation
\[
\hspace{1cm}
\diagToVar_{\diagD{\bar 0}{L}} \cdot \diagToVar_{\diagD{0}{R}} \cdot \diagToVar_{\diagD{2}{R}} = \diagToVar_{\diag{0}{2}} \cdot \diagToVar_{\diagD{\bar 0}{R}} + \diagToVar_{\diag{2}{\bar 0}} \cdot \diagToVar_{\diagD{0}{R}},
\]
which simplifies to
\[
\hspace{1cm}
\diagToVar_{\diagD{\bar 0}{L}} \cdot \diagToVar_{\diagD{2}{R}} = \diagToVar_{\diag{0}{2}} + \diagToVar_{\diag{2}{\bar 0}}.
\]

\item The compatibility degree between two centrally symmetric pairs of chords~$\theta, \delta$ is the \defn{crossing number} $\cross[\theta][\delta]$, defined as the number of times that a representative diagonal of the pair~$\delta$ crosses the chords of~$\theta$.

\item Given any initial centrally symmetric seed pseudotriangulation~$T \eqdef \{ \theta_1, \dots ,\theta_n\}$ and any chord~$\delta$, the denominator of the cluster variable~$\diagToVar_{\delta}$ with respect to the initial cluster seed~$T$ is the product of all variables in chords of~$T$ crossed by~$\delta$, see~\cite{CeballosPilaud-dvectors}.
\end{enumerate}

\enlargethispage{.4cm}
\begin{figure}[h]
	\centerline{
		\includegraphics[scale=1]{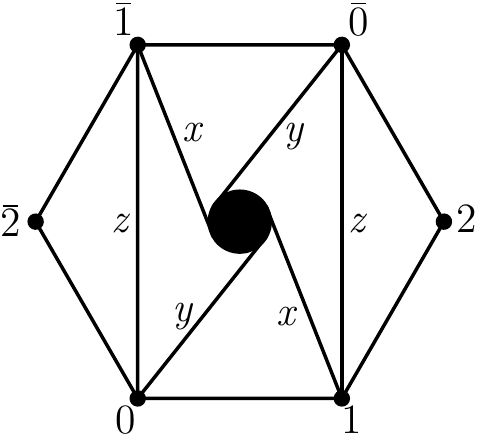}
		\qquad
		\begin{minipage}[b]{2cm} \begin{align*} \diagToVar_{\diagD{0}{L}} & = \frac{z+1}{x} \\[.3cm] \diagToVar_{\diagD{1}{L}} & = \frac{z+1}{y} \\[.3cm] \diagToVar_{\diagD{2}{R}} & = \frac{x+y}{z} \end{align*} \vspace*{.45cm} \end{minipage}
		\qquad
		\begin{minipage}[b]{2cm} \begin{align*} \diagToVar_{[0,2]} & = \frac{x+y+yz}{xz} \\[.3cm] \diagToVar_{[1,\bar 2]} & = \frac{x+y+xz}{yz} \\[.3cm] \diagToVar_{\diagD{2}{L}} & = \frac{(x+y)(z+1)}{xyz} \end{align*} \vspace*{.45cm} \end{minipage}
	}
	\caption{Cluster variables associated with centrally symmetric pairs of chords of~$\configD_3$.}
	\label{fig:exampleVariables}
\end{figure}

\begin{landscape}
\pagestyle{empty}

\begin{figure}[p]
	\vspace*{-2.5cm}
	\centerline{\includegraphics[width=1.45\textwidth]{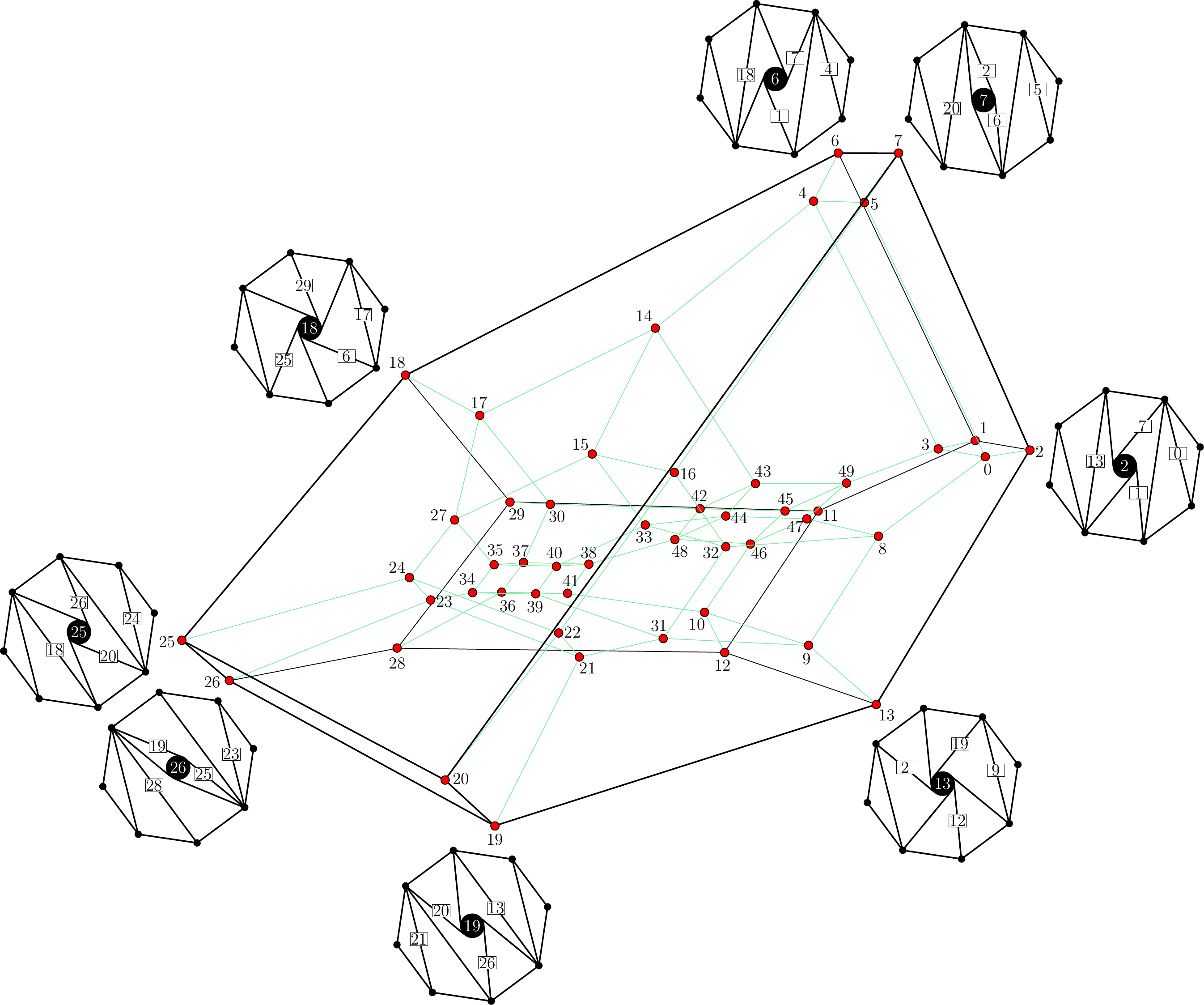}}
	\caption{The type~$D_4$ mutation graph. We have represented some of the corresponding centrally symmetric pseudotriangulations of~$\configD_4$ on this picture, while the others can be found on Figure~\ref{fig:42pseudotriangulationsHorizontal}. In each pseudotriangulation, the number at the center of the disk is its label in the mutation graph, and each pair of chords is labeled with the pseudotriangulation obtained when flipping it. The underlying graph used for the representation is a Schlegel diagram of the type~$D_4$ associahedron~\cite{ChapotonFominZelevinsky, HohlwegLangeThomas, PilaudStump-brickPolytope}.}
	\label{fig:typeD4associahedron}
\end{figure}

\begin{figure}[p]
	\vspace*{-2.5cm}
	\centerline{\includegraphics[width=1.45\textwidth]{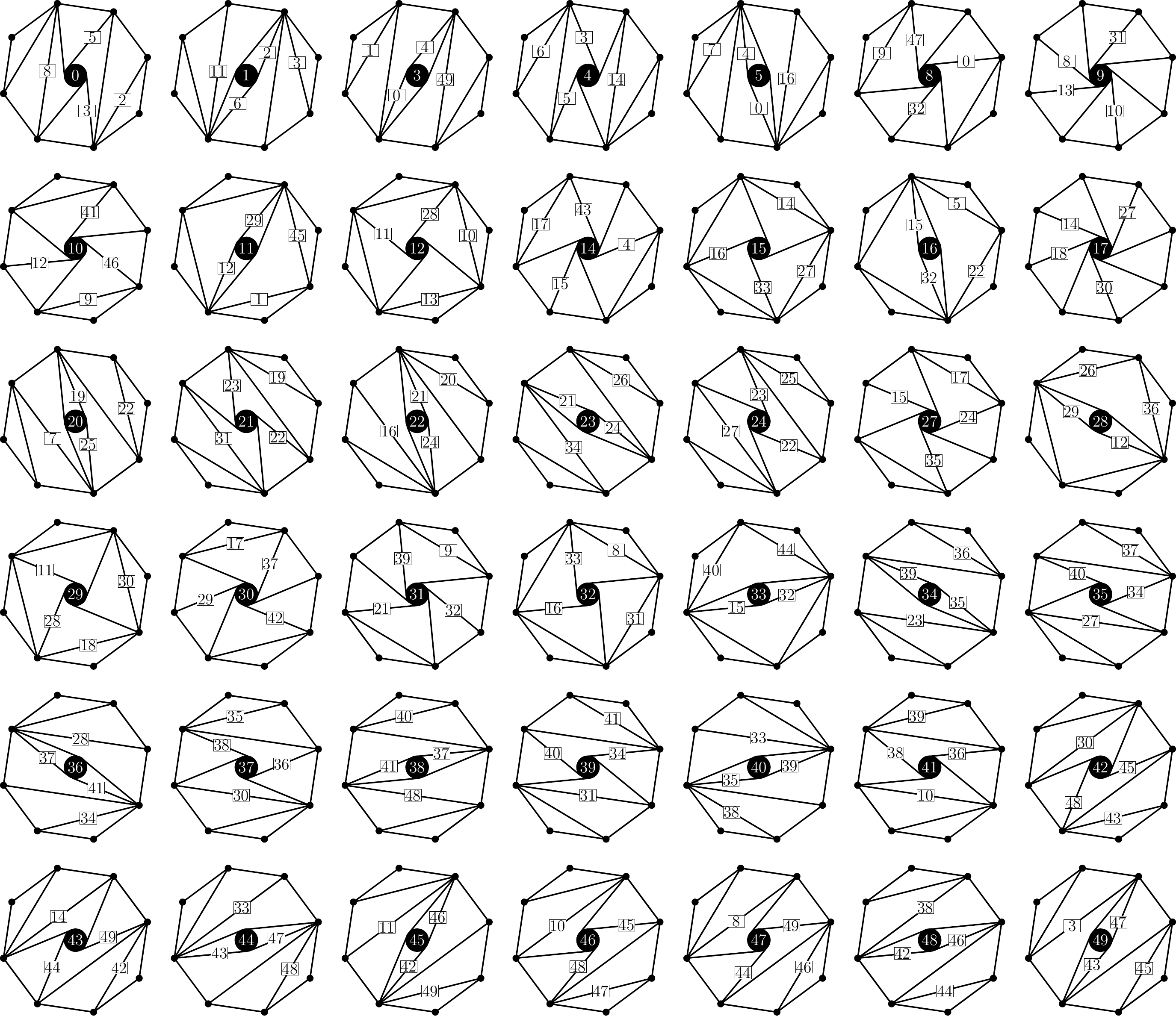}}
	\caption{The remaining~$42$ centrally symmetric pseudotriangulations of the configuration~$\configD_4$. See \fref{fig:typeD4associahedron} for the other~$8$ centrally symmetric pseudotriangulations, the mutation graph and the explanation of the labeling conventions.}
	\label{fig:42pseudotriangulationsHorizontal}
\end{figure}

\end{landscape}

\begin{remark}
\label{rem:comparisonOldModel}

\begin{figure}[t]
        \centerline{\includegraphics[scale=1]{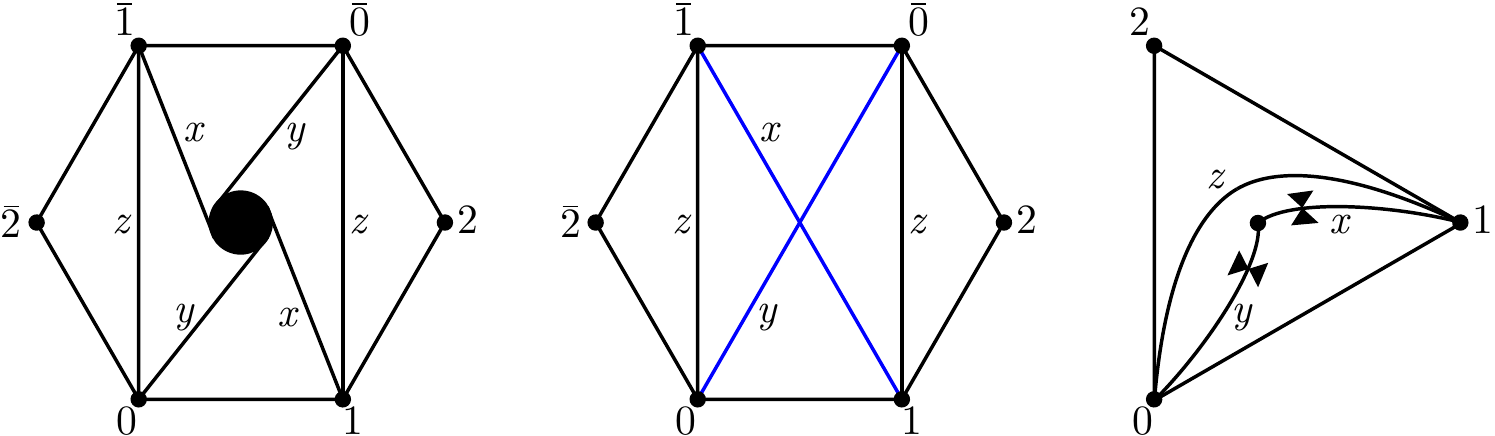}}
        \caption{The correspondence between the interpretation of type~$D_3$ clusters by centrally symmetric pseudotriangulations of~$\configD_3$ (left), by centrally symmetric triangulations of the hexagon with bicolored diagonals (middle), and by tagged triangulations of the punctured triangle (right).}
        \label{fig:differentModels}
\end{figure}

\begin{figure}[b]
        \centerline{\includegraphics[width=.48\textwidth]{typeD3associahedron}\hspace*{-.5cm}\includegraphics[width=.48\textwidth]{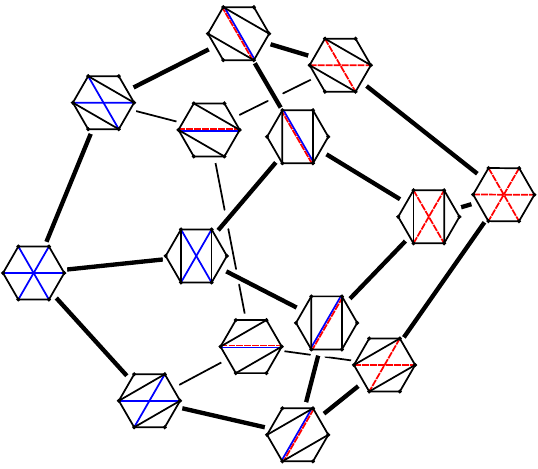}\hspace*{-.5cm}\includegraphics[width=.48\textwidth]{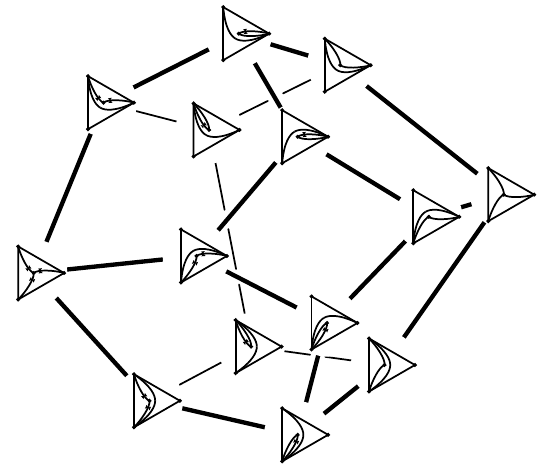}}
        \caption{The type~$D_3$ mutation graph interpreted geometrically by centrally symmetric pseudotriangulations of~$\configD_3$ (left), by centrally symmetric triangulations of the hexagon with bicolored diagonals (middle), and by tagged triangulations of the punctured triangle (right).}
        \label{fig:typeD3associahedraDifferentModels}
\end{figure}

Our geometric interpretation of type~$D$ cluster algebras slightly differs from that of S.~Fomin and A.~Zele\-vinsky in~\cite[Section~3.5]{FominZelevinsky-YSystems}\cite[Section~12.4]{FominZelevinsky-ClusterAlgebrasII}. Namely, to obtain their interpretation, we can remove the disk in the configuration~$\configD_n$ and replace the centrally symmetric pairs of chords~$\{\diagD{p}{L}, \diagD{\bar p}{L}\}$ and~$\{\diagD{p}{R}, \diagD{\bar p}{R}\}$ by long diagonals~$\diag{p}{\bar p}$ colored in red and blue respectively. Long diagonals of the same color are then allowed to cross, while long diagonals of different colors cannot. See Figures~\ref{fig:differentModels} and~\ref{fig:typeD3associahedraDifferentModels} (in color). Flips and exchange relations can then be worked out, with special rules for colored long diagonals, see~\cite[Section~3.5]{FominZelevinsky-YSystems}\cite[Section~12.4]{FominZelevinsky-ClusterAlgebrasII}.

Our interpretation can also be translated to the interpretation of S.~Fomin, M.~Shapiro and D.~Thurston~\cite{FominShapiroThurston} in terms of tagged triangulations of a convex $n$-gon with one puncture in its centre. Label the vertices of the $n$-gon in counterclockwise direction from~$0$ up to~$n-1$. If~$a\neq b$ are two vertices on the boundary, denote by~$M_{a,b}$ an arc from~$a$ to~$b$ which is homotopy equivalent to the path along the boundary from~$a$ to~$b$ in counterclockwise direction. Let~$M_{a,a}$ be the straight line connecting vertex~$a$ with the puncture, and~$M_{a,a}^{-1}$ be the same line with a tag. Two arcs~$M_{a,a}$ and~$M_{b,b}^{-1}$ are considered to cross when~$a\neq b$. Any two other arcs are considered to cross if they cross in the usual sense. A \defn{tagged triangulation} is a maximal set of non-crossing tagged arcs (up to homotopy equivalence). If~$a<b$,~$M_{a,b}$ corresponds to the pair of chords~$\diag{a}{b},\diag{\bar a}{\bar b}$ in the pseudotriangulations model. If~$a>b$, $M_{a,b}$ corresponds to the pair~$\diag{a}{\bar b},\diag{\bar a}{b}$. The line~$M_{a,a}$ is replaced by the pair~$\{\diagD{a}{L}, \diagD{\bar a}{L}\}$, and~$M_{a,a}^{-1}$ is replaced by the pair~$\{\diagD{a}{R}, \diagD{\bar a}{R}\}$. See Figures~\ref{fig:differentModels} and~\ref{fig:typeD3associahedraDifferentModels}. Exchange rules are similar to type~$A$, with special rules for tagged arcs connected to the puncture.
\end{remark}

There are different kinds of pseudotriangles, according on whether they touch the central disk~$\disk$ and the boundary of the $2n$-gon. We say that a pseudotriangle is \defn{central} if it has a corner on the disk~$\disk$, \defn{degenerate central} if it is bounded by the disk~$\disk$ and two central chords~$\diagD{p}{L}, \diagD{p}{R}$, and \defn{internal} if it contains no edge of the boundary of the $2n$-gon.

As illustrated on \fref{fig:central}, we say that a pseudotriangulation is
\begin{itemize}
\item \defn{central} if it contains two degenerate central pseudotriangles, or equivalently, if it contains exactly one left pair and one right pair of central chords, and
\item of~\defn{type left} (resp.~\defn{type right}) if all its central chords are left (resp.~right) central chords.
\end{itemize}

\begin{figure}[h]
	\centerline{
	\begin{tabular}{ccccc}
		\includegraphics[width=.18\textwidth]{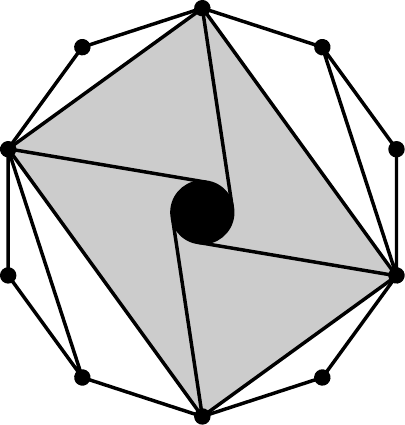} &
		\includegraphics[width=.18\textwidth]{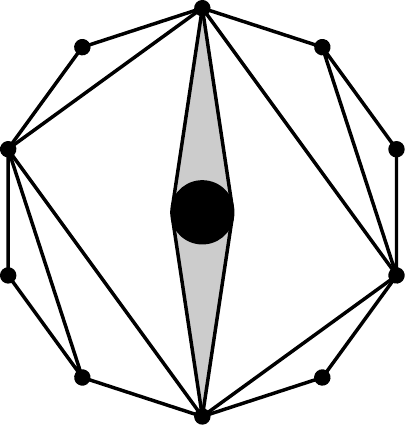} &
		\includegraphics[width=.18\textwidth]{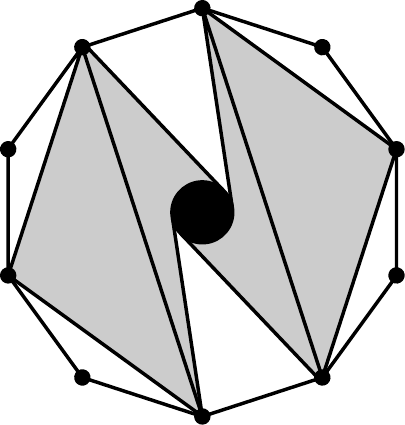} &
		\includegraphics[width=.18\textwidth]{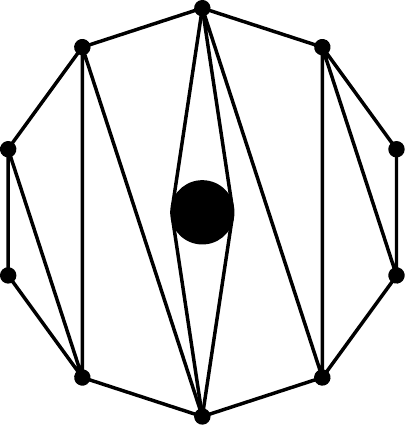} &
		\includegraphics[width=.18\textwidth]{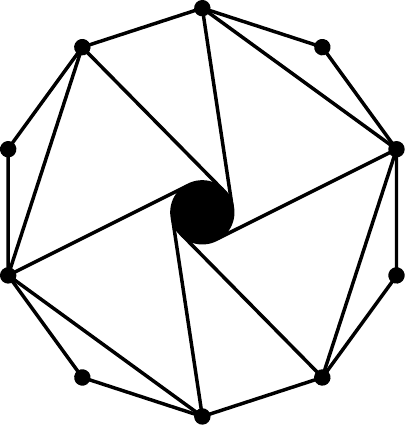} \\
		central &
		degenerate central &
		internal &
		central &
		left type \\
		pseudotriangles & 
		pseudotriangles & 
		pseudotriangles & 
		pseudotriangulation & 
		pseudotriangulation
	\end{tabular}
	}
	\caption{Different kinds of pseudotriangles and pseudotriangulations.}
	\label{fig:central}
\end{figure}


\section{Quivers}
\label{sec:quivers}

\begin{figure}[b]
	\centerline{\includegraphics[width=.9\textwidth]{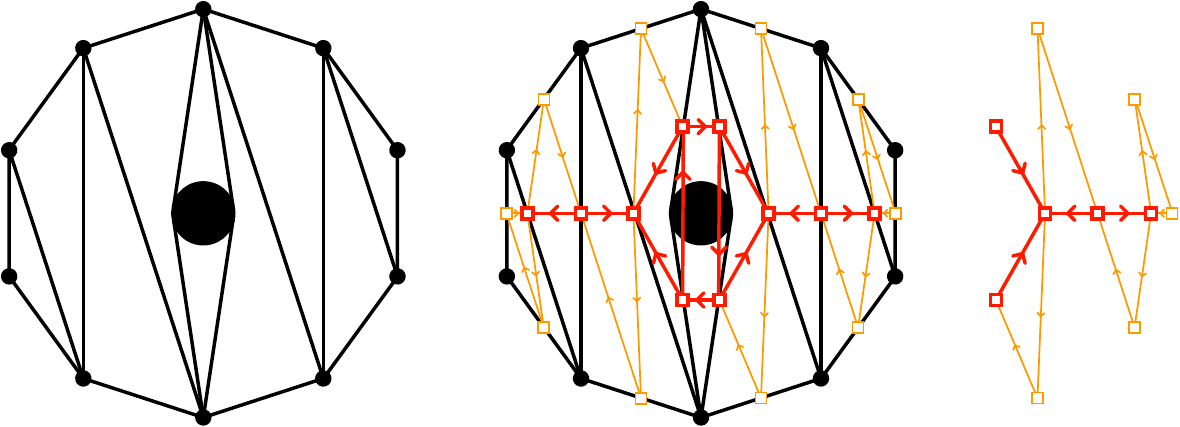}}
	\caption{A pseudotriangulation~$T$ (left), the double quiver~$\widetilde\quiver(T)$ (middle) and the quiver~$\quiver(T)$ (right).}
	\label{fig:quiver}
\end{figure}

In this section, we give a direct geometric interpretation of type~$D_n$ quivers using centrally symmetric pseudotriangulations of~$\configD_n$. This construction is very similar to type~$A$, except that triangles are replaced by pseudotriangles. An explicit example is illustrated on \fref{fig:quiver}. \fref{fig:quivers} shows several examples along a path of flips between pseudotriangulations. Given a centrally symmetric pseudotriangulation~$T$ of~$\configD_n$, we first construct a double quiver~$\widetilde\quiver(T)$ as follows:
\begin{description}
\item[nodes] middle points of each chord of~$T$;
\item[arcs] arcs connecting all chords of each side of a psuedotriangle~$\pseudotriangle$ of~$T$ to all chords of the next side of~$\pseudotriangle$ in clockwise order. The only technicality is to treat degenerate central pseudotriangles: as for exchange relations, it is natural to consider degenerate central pseudotriangles as flatten pseudotriangles, which indicates that the four central chords must form a clockwise $4$-cycle in~$\widetilde\quiver(T)$.
\end{description}
Finally, to obtain the quiver~$\quiver(T)$ of the pseudotriangulation~$T$, we fold the double quiver~$\widetilde\quiver(T)$ by central symmetry, and simplify opposite arcs (due to $4$-cycles arising when~$T$ has only $4$ central chords) and duplicated arcs (arising from pairs of centrally symmetric copies of arcs of~$\widetilde\quiver(T)$).

The reader is invited to observe in \fref{fig:quivers} the following facts concerning the quiver~$\quiver(T)$:
\begin{enumerate}[(i)]
\item The left and right star pseudotriangulations correspond to the two $n$-cycle quivers (which are empty when~$n = 2$).
\item There are two kinds of cycles in the quiver~$\quiver(T)$: one (possibly empty) cycle connecting all central chords of~$T$, and other cycles that arise from internal pseudotriangles of~$T$.
\item The quiver of a pseudotriangulation is acyclic if and only if one of the four following situations happens, where the shaded part has no internal triangles:
\begin{center}
\medskip
\includegraphics[width=.9\textwidth]{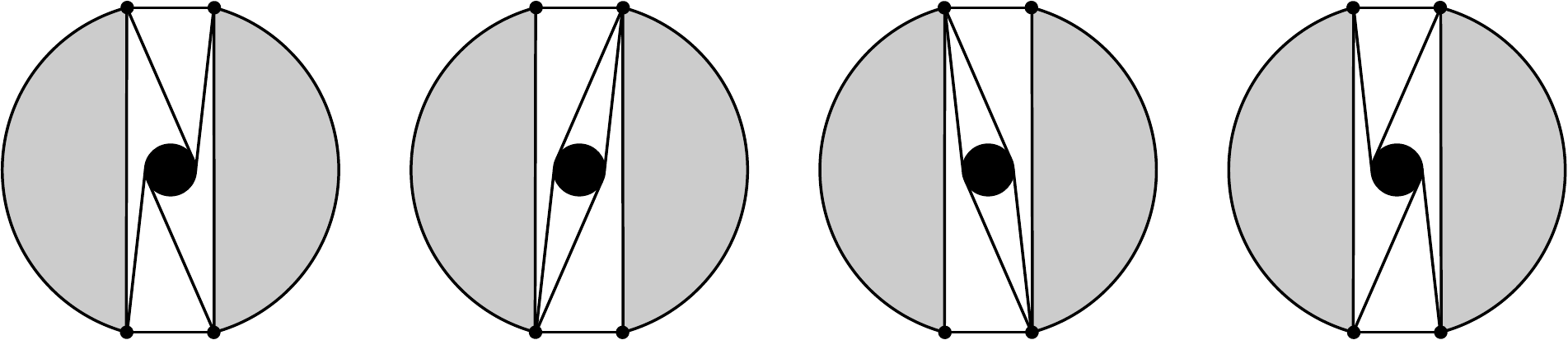}
\medskip
\end{center}
\end{enumerate}

\medskip
To show that~$\quiver(T)$ is precisely the quiver of the cluster seed corresponding to~$T$, we thus only need to show that the map~$T \to \quiver(T)$ sends flips on pseudotriangulations to mutations on~quivers.

\begin{proposition}
\label{prop:quiverMutation}
Flips in pseudotriangulations correspond to quiver mutations. More precisely, the quiver of the pseudotriangulation produced by the flip of a pair of chords~$\chi$ in~$T$ coincides with the quiver produced by mutation of the node corresponding to~$\chi$ in the quiver of~$T$.
\end{proposition}

\begin{proof}
We check this property separately on the four possible flips presented in \fref{fig:typeDflip}: the first flip is a classical flip in a quadrangle; the second flip behaves exactly as a classical flip; the third flip is similar except that one side of the classical rectangle is bended; the fourth flip is different but it indeed mutates the associated quiver.
Note that this also explains the exchange relation between variables presented in the previous section: 
\[
\Pi(\pseudoquadrangle,p,r) \cdot \Pi(\pseudoquadrangle,q,s) = \Pi(\pseudoquadrangle,p,q) \cdot \Pi(\pseudoquadrangle,r,s) + \Pi(\pseudoquadrangle,p,s) \cdot \Pi(\pseudoquadrangle,q,r).
\]
Indeed, the exchange relation is driven by the arcs of the quiver incident to the mutated variable, which in turn are determined by the sides of the pseudoquadrangle involved in the flip.
\end{proof}

\begin{remark}
\enlargethispage{.3cm}
The quiver associated to a type~$D$ cluster can be obtained from the different geometric models mentioned in Remark~\ref{rem:comparisonOldModel}. See~\cite[Section~3.5]{FominZelevinsky-YSystems}\cite[Section~12.4]{FominZelevinsky-ClusterAlgebrasII}\cite{FominShapiroThurston} for a detailed presentation of these models, and \fref{fig:quiverDifferentModels} for an illustration of the associated quivers.
\begin{figure}[h]
        \centerline{\includegraphics[scale=1]{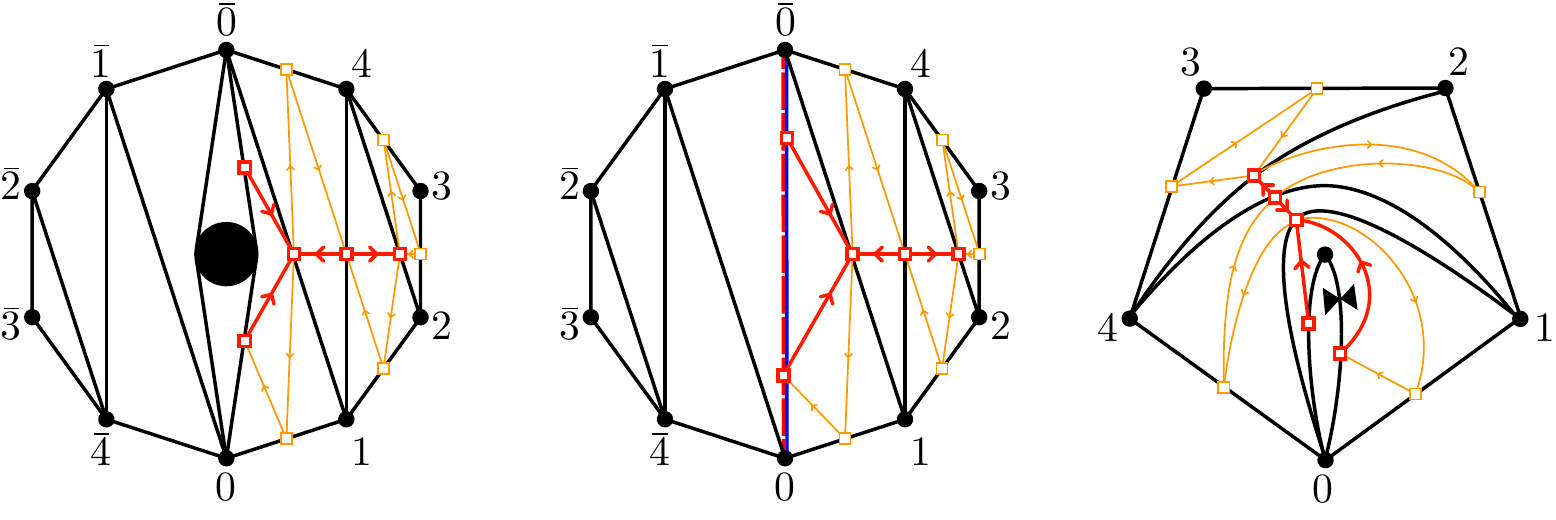}}
        \caption{The quiver associated to the centrally symmetric pseudotriangulation of~$\configD_3$, the centrally symmetric triangulation of the hexagon with bicolored diagonals, and by tagged triangulations of the punctured triangle of \fref{fig:differentModels}.}
        \label{fig:quiverDifferentModels}
\end{figure}
\begin{figure}[p]
	\centerline{\includegraphics[width=\textwidth]{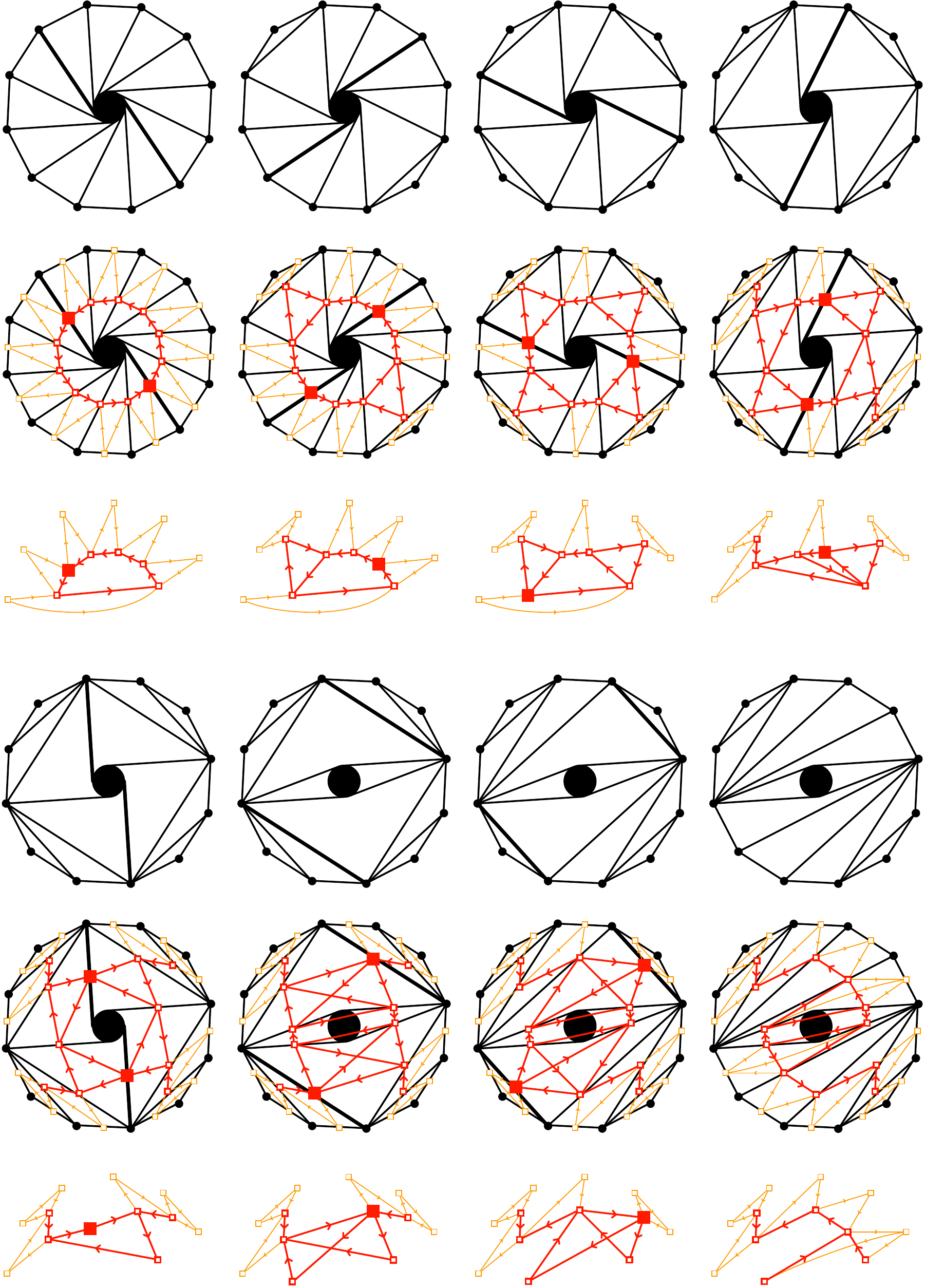}}
	\caption{A sequence of flips on centrally symmetric pseudotriangulations, and the induced mutation sequence of their associated quivers.}
	\label{fig:quivers}
\end{figure}
\end{remark}


\section{Perfect matching enumerators}
\label{sec:perfectMatchings}

This section presents a new perfect matching interpretation of cluster variables of type~$D$ based on the combinatorial model presented in Section~\ref{sec:model}. For a given initial cluster, the cluster variables can be computed in terms of weighted perfect matching enumeration of a graph after deleting two of its vertices. An alternative interpretation of cluster variables in terms of perfect matchings of snake graphs follows from the work of G.~Musiker, R.~Schiffler and L.~Williams~\cite{MusikerSchifflerWilliams} in cluster algebras from surfaces, for the particular case where the surface is a convex polygon with one puncture.  
One main difference of our method is that once the cluster seed is fixed one or two graphs serve for all cluster variables (see Remark~\ref{rem:matchingComparison}), while in~\cite{MusikerSchifflerWilliams} the snake graph needs to be recomputed depending on the cluster variable. Another perfect matching interpretation of cluster variables in classical finite types (including type $D$) is presented by G.~Musiker in~\cite{Musiker} with respect to a bipartite seed.   

\begin{figure}[b]
	\centerline{
	\begin{tabular}{ccc}
		\includegraphics[width=.25\textwidth]{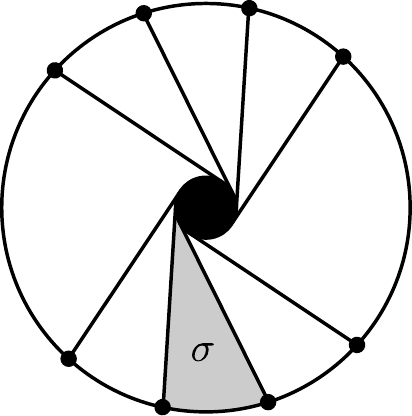} &
		\includegraphics[width=.25\textwidth]{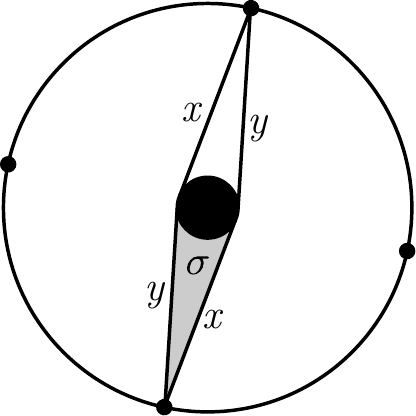} &
		\includegraphics[width=.25\textwidth]{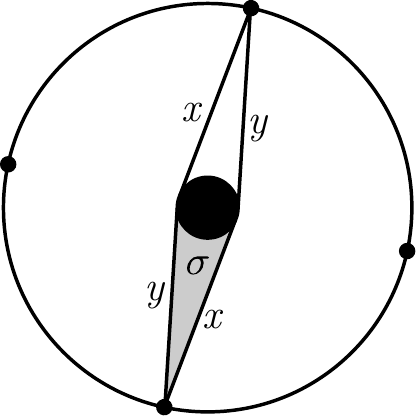} \\
		$T$ & $T$ & $T$ \\[.3cm]
		\includegraphics[width=.25\textwidth]{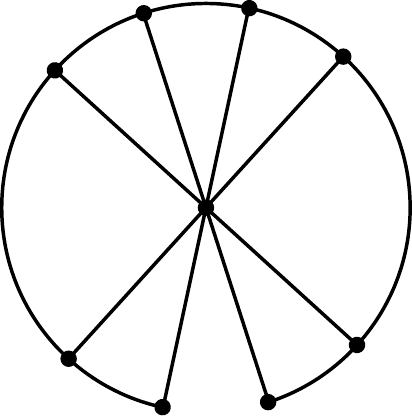} &
		\includegraphics[width=.25\textwidth]{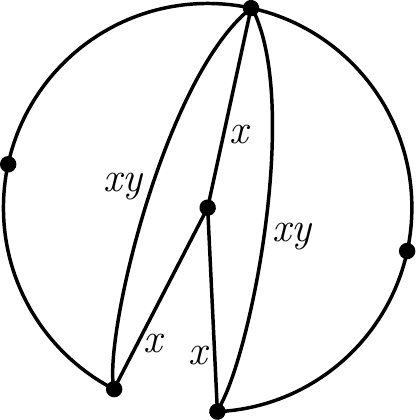} &
		\includegraphics[width=.25\textwidth]{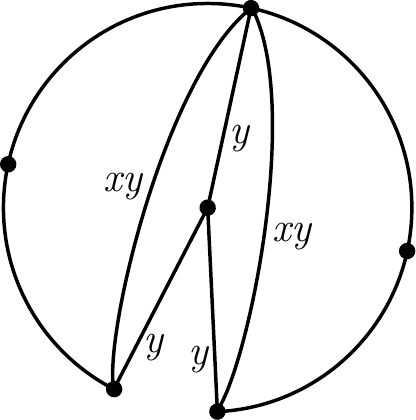} \\
		opening~$T_\sigma$ & opening $T_\sigma$ along~$x$ & opening $T_\sigma$ along~$y$
	\end{tabular}
	}
	\caption{The opening~$T_\sigma$ of a centrally symmetric pseudotriangulation~$T$ with respect to a central pseudotriangle~$\sigma$ of~$T$.}
	\label{fig:opening}
\end{figure}

Fix a centrally symmetric pseudotriangulation~$T$ of~$\configD_n$ whose internal chords are centrally symmetrically labeled with distinct variables and whose boundary edges are labeled by 1. Define the \defn{opening} of~$T$ with respect to a central pseudotriangle~$\sigma$ of~$T$ to be the triangulation~$T_\sigma$ illustrated in Figure~\ref{fig:opening} (see also Figures~\ref{fig:matchingExample1} and~\ref{fig:matchingExample2}), and call \defn{weights} the induced variable labeling of its edges. We omit the labels 1 in all figures for simplicity. Define a weighted bipartite graph~${G_\sigma=G_\sigma(T)}$ as the weighted vertex-triangle incidence graph of~$T_\sigma$. That is, $G_\sigma$ has black vertices corresponding to the vertices of~$T_\sigma$, and  white vertices corresponding to the triangles in~$T_\sigma$. A black vertex~$v$ is connected to a white vertex~$w$ if they are incident in the triangulation~$T_\sigma$, that is, if~$v$ is one of the vertices of the triangle corresponding to~$w$. In such case, the weight of the edge~$vw$ in~$G_\sigma$ is the weight of the edge of the triangle~$w$ opposite to the vertex~$v$.

\begin{figure}
	\centerline{
	\begin{tabular}{cc@{\hspace{0cm}}c}
		\includegraphics[width=.3\textwidth]{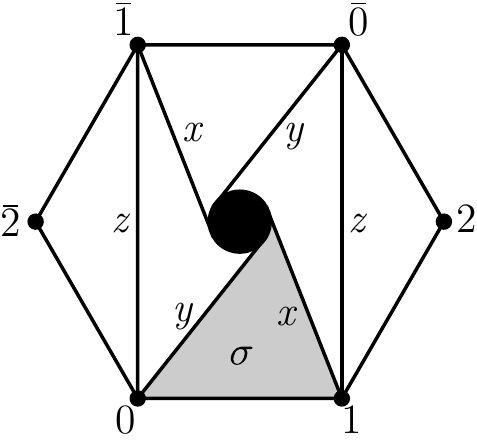} &
		\includegraphics[width=.3\textwidth]{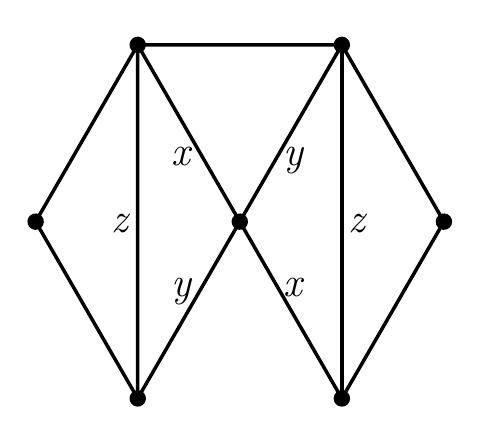} &
		\includegraphics[width=.3\textwidth]{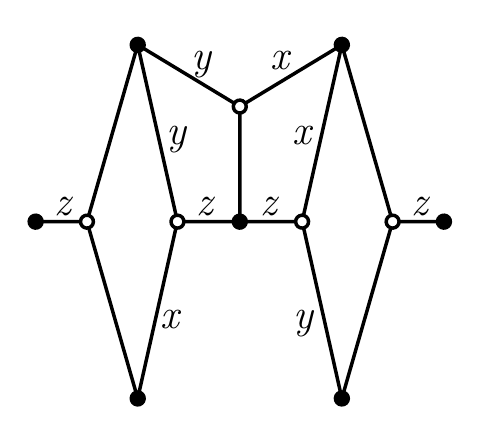} \\[-.3cm]
		$T$ & $T_\sigma$ & $G_\sigma$ \\[.1cm]
		\includegraphics[width=.3\textwidth]{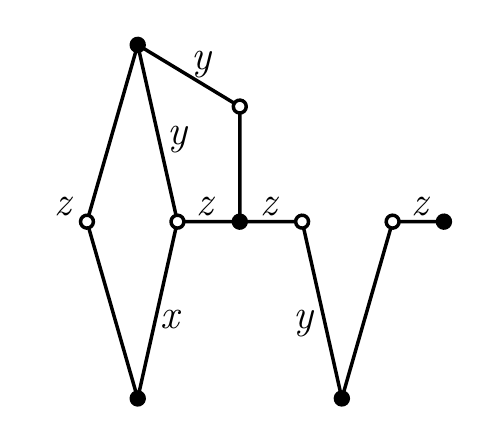} &
		\includegraphics[width=.3\textwidth]{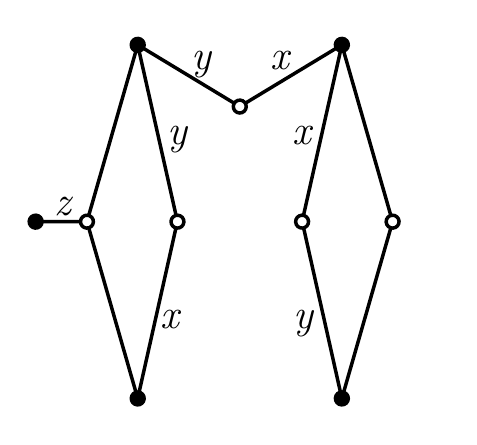} &
		\includegraphics[width=.3\textwidth]{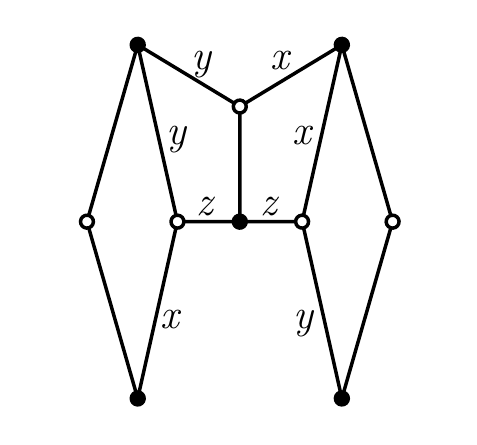} \\[-.4cm]
		$G_{\sigma,\diag{\bar 0}{\bar 2}}$ & $G_{\sigma,\diagD{2}{R}}$ & $G_{\sigma,\diagD{\bar 2}{L}}$\\[0.3cm]
		$w_{\sigma,\diag{\bar 0}{\bar 2}}=yz(x+y+yz)$ & $w_{\sigma,\diagD{2}{R}}=xyz(x+y)$ & $w_{\sigma,\diagD{\bar 2}{L}}=(x+y)(x+y+xz+yz)$\\[0.3cm]
		$m_{\sigma,\diag{\bar 0}{\bar 2}}=\displaystyle \frac{x+y+yz}{xz}$ & $m_{\sigma,\diagD{2}{R}}=\displaystyle \frac{x+y}{z}$ & $m_{\sigma,\diagD{\bar 2}{L}}=\displaystyle \frac{x+y}{z}\cdot \frac{x+y+xz+yz}{xyz}$\\[0.3cm]
		$\diagToVar_{\diag{\bar 0}{\bar 2}}=\displaystyle \frac{x+y+yz}{xz}$ & $\diagToVar_{\diagD{2}{R}}=\displaystyle \frac{x+y}{z}$ & $\diagToVar_{\diagD{\bar 2}{L}}=\displaystyle \frac{x+y+xz+yz}{xyz} $
	\end{tabular}
	}
	\caption{Examples of cluster variable computations in terms of perfect matchings of the graph~$G_{\sigma,\delta}$.}
	\label{fig:matchingExample1}
\end{figure}

\begin{figure}
	\centerline{
	\begin{tabular}{ccc}
		\includegraphics[width=.3\textwidth]{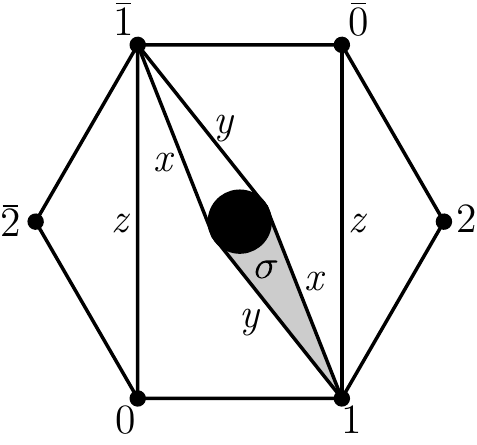} &
		\includegraphics[width=.3\textwidth]{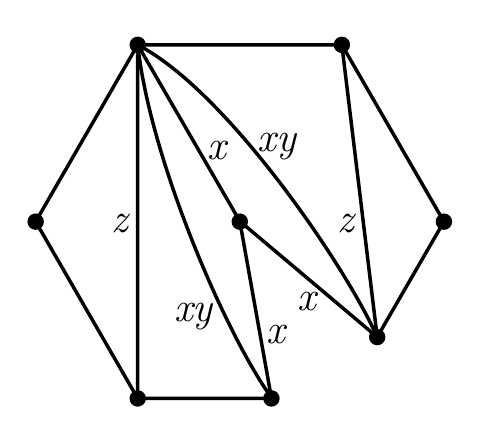} &
		\includegraphics[width=.3\textwidth]{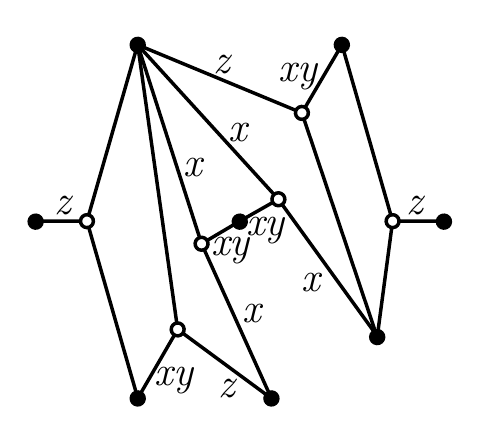} \\[-.3cm]
		$T$ & $T_\sigma$ & $G_\sigma$ \\[.1cm]
		\includegraphics[width=.3\textwidth]{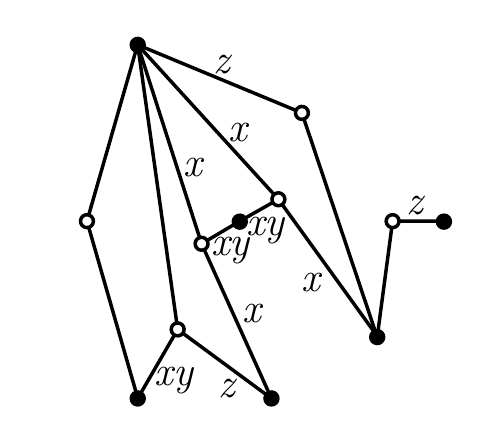} &
		\includegraphics[width=.3\textwidth]{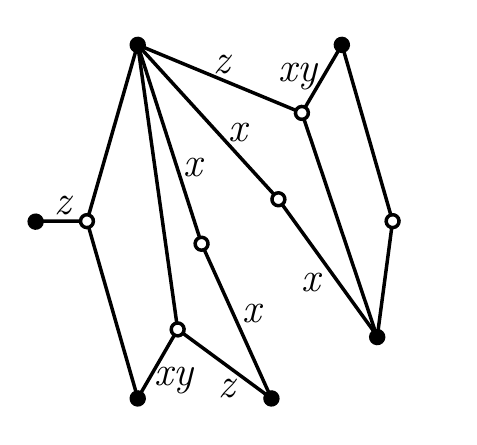} &
		\includegraphics[width=.3\textwidth]{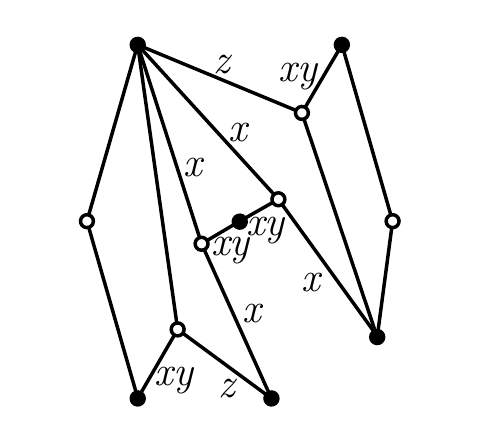} \\[-.4cm]
		$G_{\sigma,\diag{\bar 0}{\bar 2}}$ & $G_{\sigma,\diagD{2}{R}}$ & $G_{\sigma,\diagD{\bar 2}{L}}$\\[0.3cm]
		$w_{\sigma,\diag{\bar 0}{\bar 2}}=x^2yz\big(xy+(z+1)^2\big)$ & $w_{\sigma,\diagD{2}{R}}=x^3yz(xy+z+1)$ & $w_{\sigma,\diagD{\bar 2}{L}}=x^2y(xy+z+1)^2$\\[0.3cm]
		$m_{\sigma,\diag{\bar 0}{\bar 2}}=\displaystyle \frac{xy+(z+1)^2}{xyz}$ & $m_{\sigma,\diagD{2}{R}}=\displaystyle \frac{xy+z+1}{yz}$ & $m_{\sigma,\diagD{\bar 2}{L}}=\displaystyle \frac{(xy+z+1)^2}{xyz^2}$\\[0.3cm]
		$\diagToVar_{\diag{\bar 0}{\bar 2}}=\displaystyle \frac{xy+(z+1)^2}{xyz}$ & $\diagToVar_{\diagD{2}{R}}=\displaystyle \frac{xy+z+1}{yz}$ & $\diagToVar_{\diagD{\bar 2}{L}}=\displaystyle \frac{xy+z+1}{xz}$	\end{tabular}
	}
	\caption{Examples of cluster variable computations in terms of perfect matchings of the graph~$G_{\sigma,\delta}$.}
	\label{fig:matchingExample2}
\end{figure}

For each chord~$\delta$, we will define a graph~$G_{\sigma,\delta}=G_{\sigma,\delta}(T)$ obtained from~$G_\sigma$ by deleting two black vertices and all their incident edges. The cluster variable~$\diagToVar_\delta$ associated to the centrally symmetric pair~$\{\delta,\delta'\}$ will be determined in terms of perfect matchings of~$G_{\sigma,\delta}$. Figures~\ref{fig:matchingExample1} and~\ref{fig:matchingExample2} illustrate some examples of~$T_\sigma$,~$G_\sigma$,~$G_{\sigma,\delta}$ and~$\diagToVar_\delta$.

We say that the opening~$T_\sigma$ is of \defn{type left} (resp.~\defn{right}) if~$T$ is of type left (resp.~right) or~$T$ is a central pseudotriangulation which is opened along its left (resp.~right) central chord. 
The graph~$G_{\sigma,\delta}$ is obtained from~$G_\sigma$ by deleting
\begin{itemize}
\item the endpoints of~$\delta$ if~$\delta$ is not a central chord or a central chord of the same type as~$T_\sigma$,~and
\item the vertex of the $2n$-gon which is an endpoint of~$\delta$ and its opposite vertex if~$\delta$ is a central chord of different type than~$T_\sigma$.
\end{itemize} 
Let~$w_{\sigma,\delta}=w_{\sigma,\delta}(T)$ be the sum of the weights of all perfect matchings of the graph~$G_{\sigma,\delta}$, where the weight of a perfect matching is the product of the weights of its edges. Define~$m_{\sigma,\delta}=m_{\sigma,\delta}(T)$ as~$w_{\sigma,\delta}$ divided by the product of the weights of all internal diagonals in~$T_\sigma$. 

\begin{theorem}
Let~$T$ be a centrally symmetric pseudotriangulation and~$\delta$ be a chord of~$\configD_n$. For any central triangle~$\sigma$ which is not crossed by~$\delta$, the cluster variable $\diagToVar_\delta$ associated to the centrally symmetric pair~$\{\delta, \bar\delta\}$ is determined by
\[
\widetilde \diagToVar_\delta = m_{\sigma,\delta},
\] 
where   
\[
\widetilde \diagToVar_\delta =
\begin{cases}
   \diagToVar_\delta &\text{ if } \delta \text{ is not a central chord},\\
  \diagToVar_\delta &\text{ if } \delta \text{ is a central chord of the same type as~$T_\sigma$},\\
  \diagToVar_{\diagD{p}{L}} \diagToVar_{\diagD{p}{R}} &\text{ if } \delta \text{ is a central chord incident to~$p$ of type different than~$T_\sigma$}.\\
\end{cases}
\]
\end{theorem}

\begin{proof}
The proof of this result is identical to the proof of~\cite[Theorem 2.1]{Prop} in type~$A$. The result follows directly from the following three main steps.

(1) $m_{\sigma,\delta}=1$ if~$\delta$ is a boundary edge.

(2) $m_{\sigma,\delta}=\diagToVar_\delta$ if~$\delta$ is a chord of the pseudotriangulation~$T$.

(3) Consider a pseudoquadrangle~$\pseudoquadrangle$ not crossing the central triangle~$\sigma$, and cyclically label its vertices by~$p,q,r,s$. Denote by~$y_{i,j}=\Pi(\pseudoquadrangle,i,j)$ for $i, j \in \{p,q,r,s\}$ according to the geometric model in Section~\ref{sec:model}. These variables satisfy the relation 
\[
y_{p,r}y_{q,s} = y_{p,q}y_{r,s} + y_{p,s}y_{q,r}.
\]
Now, for every pair~$i,j$ of black vertices in the graph~$G_{\sigma,\delta}$, denote by~$\widetilde w_{i,j}$ the sum of the weights of all perfect matchings of the graph obtained by deleting the vertices~$i$ and~$j$ from~$G_{\sigma,\delta}$, and define~$\widetilde m_{i,j}$ as~$\widetilde w_{i,j}$ divided by the product of the weights of all internal diagonals in~$T_\sigma$. The main ingredient of this step is to show that
\[
y_{i,j} = \widetilde m_{i,j}
\]
for every pair $i,j$ of corners of a pseudoquadrangle not crossing~$\sigma$. 
Similarly as in the proof of~\cite[Theorem 2.1]{Prop}, we use a graph theoretic lemma by E.~Kuo~\cite[Theorem~2.5]{Kuo}:

\para{Condensation Lemma} Let~$G$ be a (weighted) bipartite planar graph with 2 more black vertices than white vertices. If $p,q,r,s$ are black vertices that appear in cyclic order on a face of~$G$, then 
\[
w(p,r)w(q,s) = w(p,q)w(r,s) + w(p,s)w(q,r),
\]
where $w(i,j)$ denotes the (weighted) number of perfect matchings of the graph obtained from~$G$ by deleting the vertices~$i$ and~$j$ with all their incident edges. 

\medskip
After the opening of~$T$, the triangulation~$T_\sigma$ is outer-planar, so that all black vertices of~$G_\sigma$ lie on its external face. Therefore, the condensation lemma applies to the graph~$G_\sigma$, and leads to
\[
\widetilde m_{p,r}\widetilde m_{q,s} = \widetilde m_{p,q}\widetilde m_{r,s} + \widetilde m_{p,s}\widetilde m_{q,r}.
\]
This relation together with (1) and (2) imply $y_{i,j} = \widetilde m_{i,j}$ for every pair $i,j$ of corners of a pseudoquadrangle not crossing~$\sigma$ as desired. One interesting case is when the pseudoquadrangle is tangent to the disk~$\disk$, and $i,j$ are the corners of this tangent pseudoline. In this case,~$y_{i,j}=\diagToVar_{\diagD{p}{L}}\diagToVar_{\diagD{p}{R}}$ for a central chord~$\delta$ incident to~$p$ of type different than~$T_\sigma$, and~$\widetilde m_{i,j}=m_{\sigma,\delta}$.  The case when the pseudotriangle is not tangent to the central circle implies the two remaining cases.
\end{proof}

\begin{remark}
Labeling symmetrically the boundary edges of the $2n$-gon by frozen variables, the results of this section also apply to cluster algebras of type~$D_n$ with coefficients.
\end{remark}

\begin{remark}\label{rem:matchingComparison}
Motivated by the positivity conjecture of cluster algebras, many authors have found explicit Laurent expansion formulas for cluster variables with respect to a cluster seed. Our interpretation follows the lines of G.~Carroll and G.~Price's computation of cluster variables in type~$A$~\cite{CarrollPrice}, which is presented in unpublished work by J.~Propp in~\cite{Prop}. Cluster expansion formulas in type $A$ are also presented by R.~Schiffler in~\cite{Schiffler-clusterExpansion}. In \cite{Musiker}, G.~Musiker presents cluster expansion formulas for cluster algebras of finite classical types (including type $D$) with respect to a bipartite seed. H.~Thomas and R.~Schiffler found expansion formulas for cluster algebras from surfaces without punctures, with coefficients associated to the boundary of the surface~\cite{SchifflerThomas}. This work was generalized by R.~Schiffler in~\cite{Schiffler} where arbitrary coefficient system is considered. An alternative formulation of the results in~\cite{Schiffler} in terms of perfect matchings is presented in~\cite{MusikerSchiffler}. This approach was generalized to cluster algebras from arbitrary surfaces (allowing punctures) by G.~Musiker, R.~Schiffler and L.~Williams in~\cite{MusikerSchifflerWilliams}, where the authors give a precise description of cluster variables with respect to any cluster seed in terms of perfect matchings of snake graphs. The particular case where the surface is given by a convex $n$-gon with one puncture gives rise to a cluster algebra of type $D_n$ (see Remark~\ref{rem:comparisonOldModel}). In Figure~\ref{fig:matchingExample1_msw} and Figure~\ref{fig:matchingExample2_msw}, we illustrate the analogous computation of the cluster variables obtained in Figure~\ref{fig:matchingExample1} and Figure~\ref{fig:matchingExample2}, using the corresponding surface model and the snake graphs in~\cite{MusikerSchifflerWilliams}. In these two figures, the tagged arcs to the puncture are replaced by usual arcs to the puncture, while non-tagged arcs are replaced by loops. The variable associated to a loop is the product of the corresponding tagged and non-tagged arcs. 
We remark that the graphs we obtain are very different to the snake graphs in general. As in the snake graphs, our graphs are obtained by gluing together tiles (4-gons) along their sides, but a vertex may be contained in many tiles while in the snake graphs it is contained in at most three. The second main difference of our description is that once the cluster seed is fixed any cluster variable can be computed from one of two graphs after deleting two of its vertices, while in~\cite{MusikerSchifflerWilliams}, the snake graph depends on the cluster variable that is being computed (see Figures~\ref{fig:matchingExample1},~\ref{fig:matchingExample2},~\ref{fig:matchingExample1_msw} and~\ref{fig:matchingExample2_msw}). The two graphs in our description can be obtained by choosing (one or) two internal triangles $\sigma$ and $\sigma'$ such that any centrally symmetric pair of chords has a representative that does not cross one of the triangles. All cluster variables can then be obtained in terms of weighted perfect matching enumeration on the graph $G_\sigma$ or~$G_{\sigma'}$ after deleting two of its vertices.    

\begin{figure}[p]
	\centerline{
	\begin{tabular}{cc@{\hspace{0cm}}c}
		\includegraphics[width=.3\textwidth]{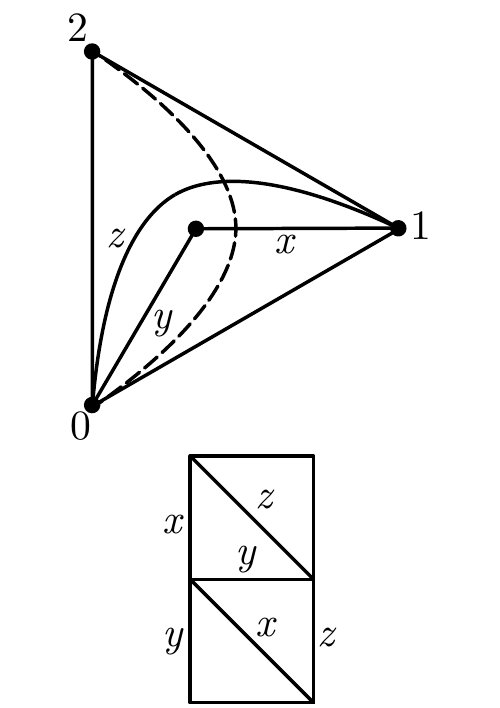} &
		\includegraphics[width=.3\textwidth]{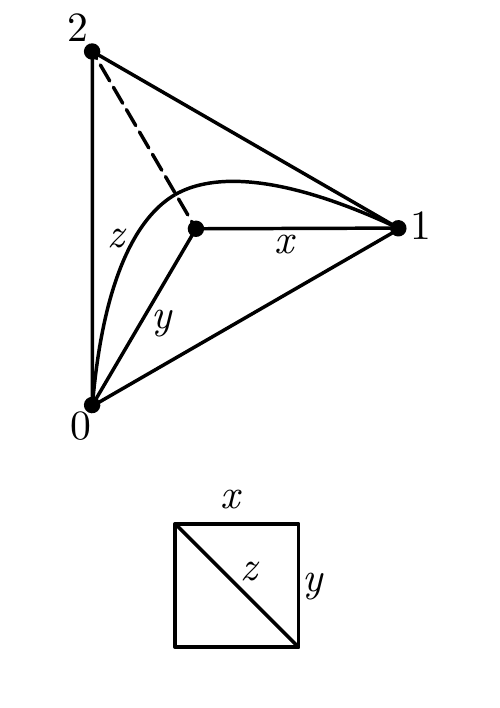} &
		\includegraphics[width=.3\textwidth]{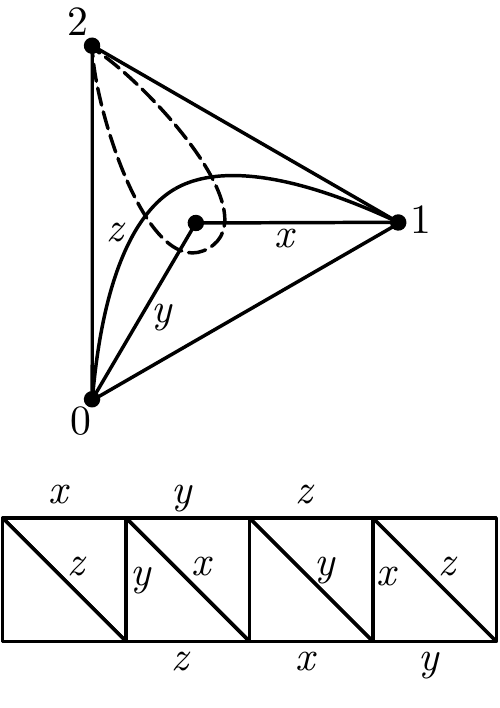} \\[.1cm]
		\multirow{2}{*}{$\diagToVar_{\diag{\bar 0}{\bar 2}} = \chi_{M_{0,2}} = \displaystyle \frac{x+y+yz}{xz}$} & \multirow{2}{*}{$\diagToVar_{\diagD{2}{R}} = \chi_{M_{2,2}^{-1}} = \displaystyle \frac{x+y}{z}$} & $\diagToVar_{\diagD{\bar 2}{R}} \cdot \diagToVar_{\diagD{\bar 2}{L}}= \chi_{M_{2,2}^{-1}} \cdot \chi_{M_{2,2}}$ \\ & & $\qquad\qquad = \displaystyle \frac{(x+y)(x+y+xz+yz)}{xyz^2}$
	\end{tabular}
	}
	\caption{G.~Musiker, R.~Schiffler and L.~William's computation of the cluster variables of type~$D_3$ from Figure~\ref{fig:matchingExample1} using the surface model and the corresponding snake graphs~\cite{MusikerSchifflerWilliams}.}
	\label{fig:matchingExample1_msw}
\end{figure}
 
\begin{figure}[p]
	\vspace{.3cm}
	\centerline{
	\begin{tabular}{lll}
		\includegraphics[width=.3\textwidth]{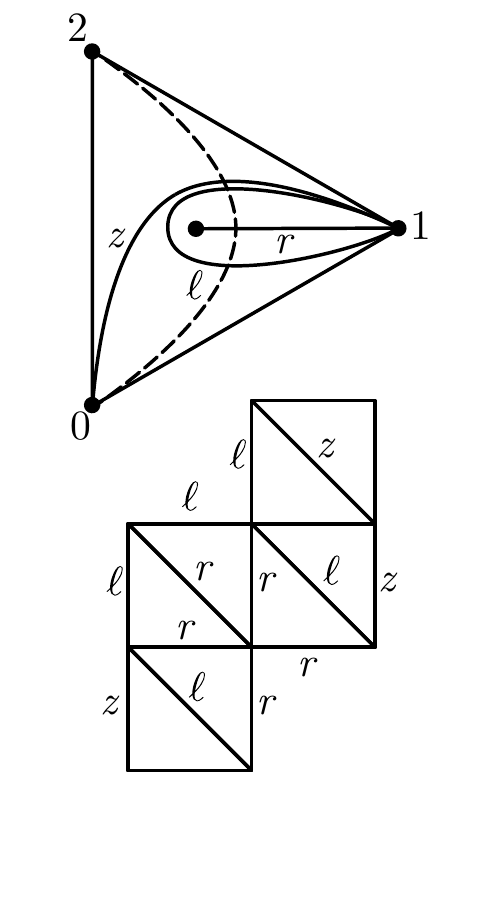} &
		\includegraphics[width=.3\textwidth]{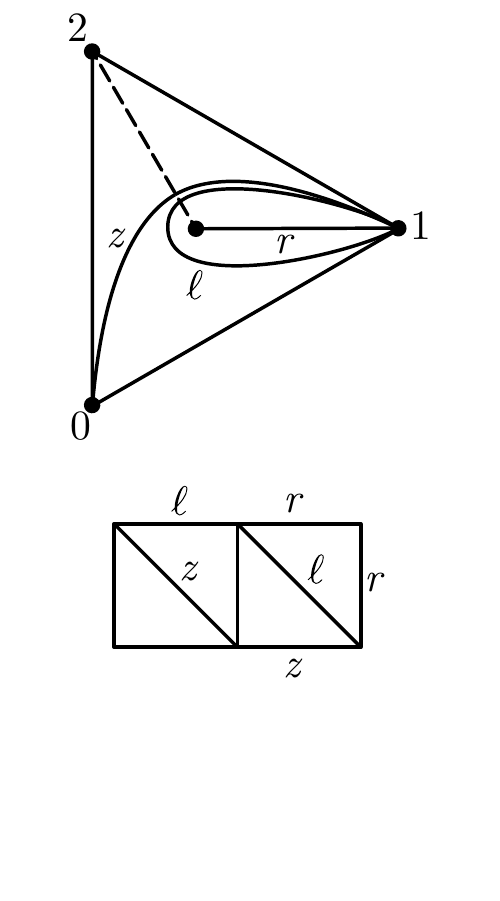} &
		\includegraphics[width=.3\textwidth]{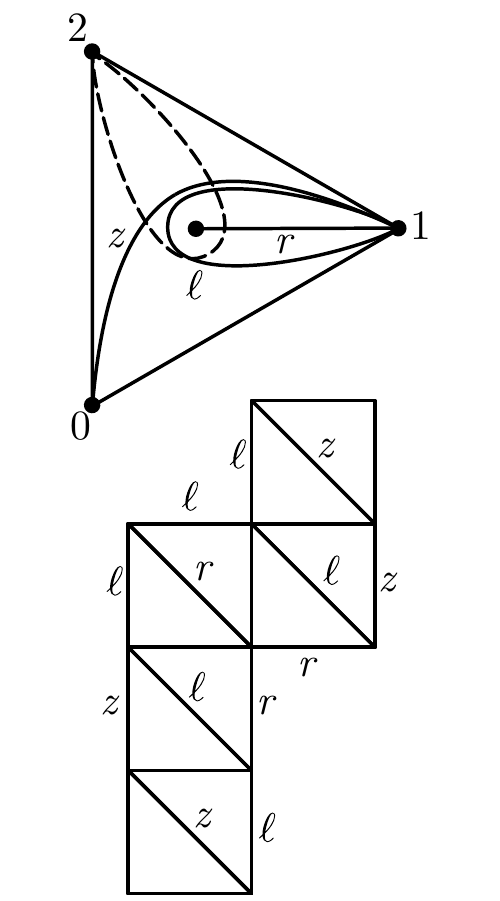} \\[.1cm]
		$\diagToVar_{\diag{\bar 0}{\bar 2}} = \chi_{M_{0,2}}$ & $\diagToVar_{\diagD{2}{R}} = \chi_{M_{2,2}^{-1}}$ & $\diagToVar_{\diagD{\bar 2}{R}} \cdot \diagToVar_{\diagD{\bar 2}{L}}= \chi_{M_{2,2}^{-1}} \cdot \chi_{M_{2,2}}$ \\
		$\quad = \displaystyle \frac{r\ell \big( \ell + (z+1)^2 \big)}{r \ell^2 z}$ & $\quad = \displaystyle \frac{\ell r + rz + r}{\ell z}$ & $\qquad = \displaystyle \frac{(\ell+z+1)^2}{\ell z^2}$ \\[.2cm] 
		$\quad = \displaystyle \frac{xy+(z+1)^2}{xyz}$ & $\quad = \displaystyle \frac{xy+z+1}{yz}$ & $\qquad = \displaystyle \frac{(xy+z+1)^2}{xyz^2}$
	\end{tabular}
	}
	\caption{G.~Musiker, R.~Schiffler and L.~William's computation of the cluster variables of type~$D_3$ from Figure~\ref{fig:matchingExample2} using the surface model and the corresponding snake graphs~\cite{MusikerSchifflerWilliams}. Note that $r=x$ and $\ell=xy$ is the product of the cluster variables of the two tagged and non-tagged arcs corresponding to the loop.}
	\label{fig:matchingExample2_msw}
\end{figure}
\end{remark}


\section{Polytopes}
\label{sec:polytopes}

Let~$P$ be a set of~$p$ points in general position (\ie no three on a line) in the plane with $h$ boundary and~$i$ interior points. A \defn{(pointed) pseudotriangulation} of~$P$ is a maximal set of edges connecting points of~$P$ that is crossing-free and pointed (any vertex is adjacent to an angle wider than~$\pi$). We refer to~\cite{RoteSantosStreinu-survey} for a survey on pseudotriangulations and their properties. Using rigidity properties of pseudotriangulations and a polyhedron of expansive motions, G.~Rote, F.~Santos and I.~Streinu showed in~\cite{RoteSantosStreinu-pseudotriangulationPolytope} that the flip graph on (pointed) pseudotriangulations of~$P$ can be realized as the graph of a~$(h+2i-3)$-dimensional polytope~$\Pseudo(P)$, called the \defn{pseudotriangulation polytope}. $3$- and $4$-dimensional examples are illustrated on \fref{fig:pseudotriangulationPolytopes}.

\begin{figure}[h]
	\centerline{\includegraphics[width=.95\textwidth]{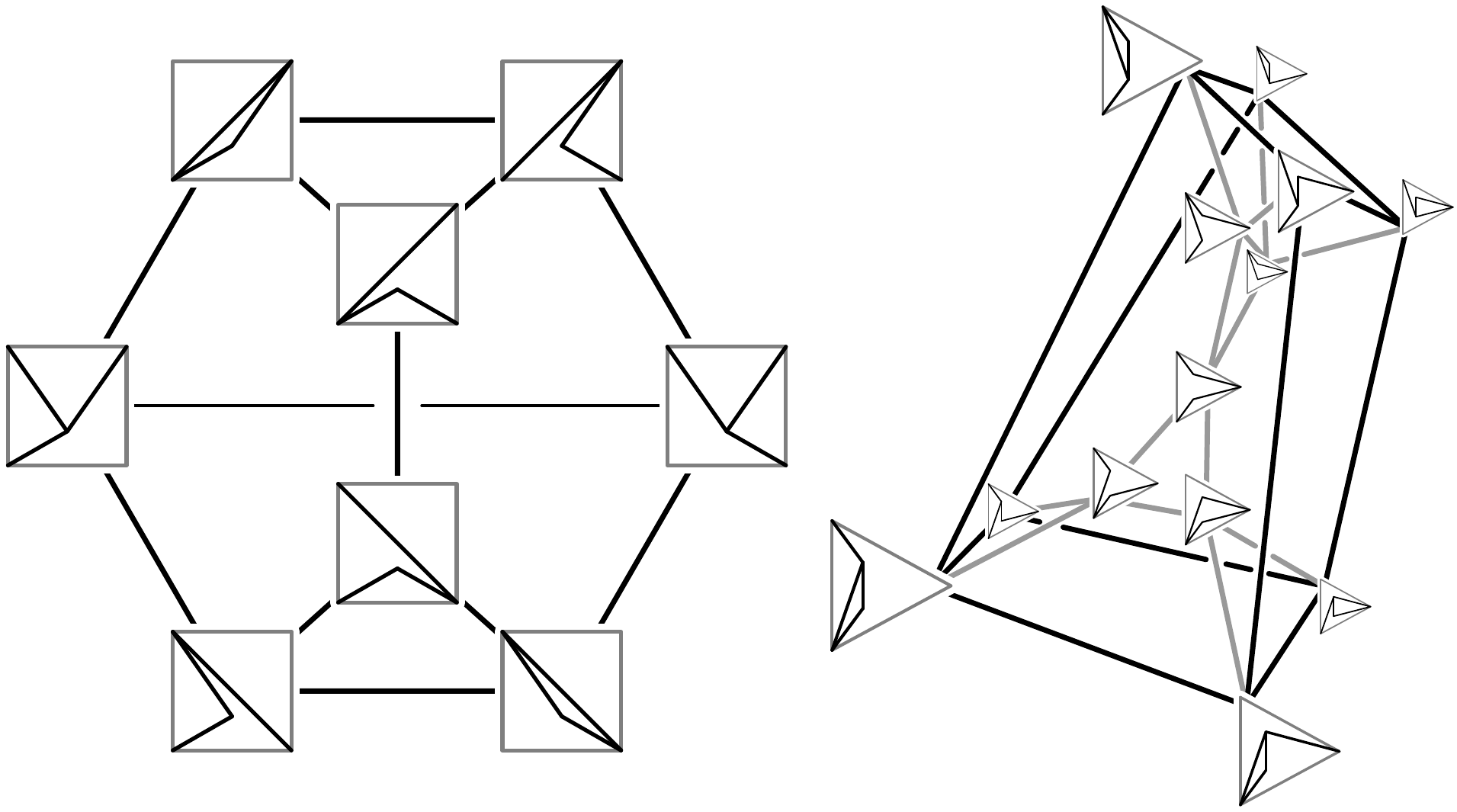}}
	\vspace{-.3cm}
	\caption{Two pseudotriangulation polytopes~\cite{RoteSantosStreinu-pseudotriangulationPolytope}.}
	\label{fig:pseudotriangulationPolytopes}
\end{figure}

In combination to our model, this construction provides polytopal realizations of the type~$D$ associahedra. Indeed, observe that:
\begin{enumerate}
\item One can model the pseudotriangulations of the configuration~$\configD_n$ by pseudotriangulations of the configuration~$\configD_n^\bullet$ obtained by replacing the disk~$D$ by the vertices of a small square~$\square$ centered at the origin. The flip graph on pseudotriangulations of~$\configD_n$ then coincides with the flip graph on pseudotriangulations of~$\configD_n^\bullet$ containing a fixed triangulation of the square~$\square$.
\item The coordinates given in the construction of the pseudotriangulation polytope~\cite{RoteSantosStreinu-pseudotriangulationPolytope} can be chosen such that the vertices of~$\Pseudo(\configD_n^\bullet)$ corresponding to centrally symmetric pseudotriangulations of~$\configD_n^\bullet$ all belong to an affine subspace. The convex hull of these vertices then realizes the flip graph on centrally symmetric pseudotriangulations of~$\configD_n^\bullet$. A face of this polytope corresponding to a fixed triangulation of the square~$\square$ gives a realization of the flip graph on centrally symmetric pseudotriangulations of~$\configD_n$, and thus a realization of the type~$D_n$ associahedron.
\end{enumerate}


\section{Connection to subword complexes}
\label{sec:subwordComplex}

Let~$\tau_0, \tau_1, \dots, \tau_{n-1}$ denote the simple generators of the Coxeter group of type~$D_n$ according to the following labeling of the Coxeter graph:

\begin{center}
\medskip
\includegraphics[scale=1]{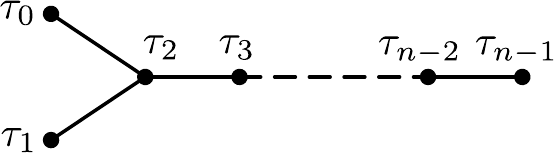}
\medskip
\end{center}

A \defn{Coxeter element} is an element of the group obtained by multiplying the generators in any given order. Fix a Coxeter element~$c$ and a reduced expression~$\sqc$ of~$c$. Let~$\Qc \eqdef \cwoc$ be the word formed by the concatenation of~$\sqc$ with the \defn{$\sqc$-sorting word} for~$\wo$ (\ie the lexicographically first reduced expression of~$\wo$ in~$\sqc^\infty$).
In the particular case of type~$D_n$, $\wo(c)=\sqc^{n-1}$ if~$\tau_0$ and~$\tau_1$ are consecutive in~$\sqc$ when considered up to commutation of consecutive commuting letters. Otherwise,~$\wo(c)$ is equal to the word obtained from~$\sqc^{n-1}$ by replacing the last appearance of~$\tau_0$ by~$\tau_1$ if~$\tau_0$ appears after~$\tau_1$ in~$\sqc$, or by replacing the last appearance of~$\tau_1$ by~$\tau_0$ if~$\tau_1$ appears after~$\tau_0$ in~$\sqc$.
Denote by~$m$ the length of~$\Qc$ and consider the rotation~$\tausqc: [m] \longrightarrow [m]$ on the positions in the word~$\Qc$ defined as follows. If~$q_i=s$, then~$\tausqc(i)$ is defined as the position in~$\Qc$ of the next occurrence of~$s$ if possible, and as the position of the first occurrence of~$\wo s \wo$ otherwise.

We now present an explicit bijection~$\posToDiagc$ between positions in~$\Qc$ and centrally symmetric pairs of chords of~$\configD_n$, which will enable us to characterize centrally symmetric pseudotriangulations of~$\configD_n$ in terms of reduced expressions of~$\wo$ in the word~$\Qc$. We keep labeling the vertices of the $2n$-gon counter-clockwise from~$0$ to~$2n-1$, and denote $\bar p \eqdef p + n \modulo{2n}$ for $p \in \{0, \dots, 2n-1\}$. 

\begin{figure}[h]
	\centerline{\includegraphics[scale=.85]{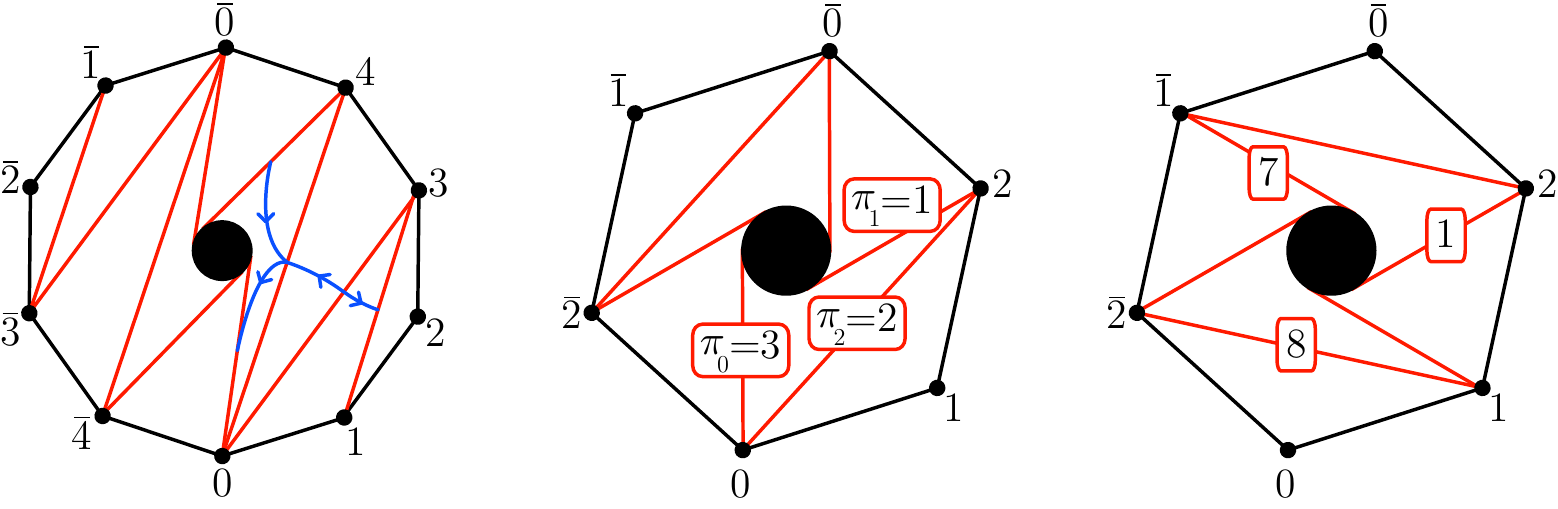}}
	\caption{The centrally symmetric ``accordion'' pseudotriangulation~$Z_c$ for ${c = \tau_3 \tau_4 \tau_0 \tau_2 \tau_1}$ (left) and for ${c = \tau_1 \tau_2 \tau_0}$ (middle) and the pseudotriangulation ${T = \posToDiag_{\tau_1\tau_2\tau_0}(\{1,7,8\}) = \{2^{\textsc l}, \bar 2^{\textsc l}, 1^{\textsc l}, \bar 1^{\textsc l}, \diag{1}{\bar 2}, \diag{\bar 1}{2}\}}$ (right).}
	\label{fig:typeD}
\end{figure}

The bijection~$\posToDiagc$ is defined as follows. Denote by~$\pi_i$ the position of~$\tau_i$ in the Coxeter element~$\sqc$. The positions~$\pi_0$ and~$\pi_1$ are sent to
$$\posToDiagc(\pi_0) = \begin{cases} \diagD{0}{L} \cup \diagD{\bar 0}{L} & \text{if } \pi_0 > \pi_2, \\ \diagD{(n-1)}{R} \cup \diagD{\overline{(n-1)}}{R} & \text{if } \pi_0 < \pi_2, \end{cases}$$
$$\posToDiagc(\pi_1) = \begin{cases} \diagD{0}{R} \cup \diagD{\bar 0}{R} & \text{if } \pi_1 > \pi_2, \\ \diagD{(n-1)}{L} \cup \diagD{\overline{(n-1)}}{L} & \text{if } \pi_1 < \pi_2, \end{cases}$$
and the positions~$\pi_2, \cdots, \pi_{n-1}$ are sent to $\posToDiagc(\pi_i) \eqdef \diag{p_i}{q_i} \cup \diag{\bar p_i}{\bar q_i}$, where
\begin{align*}
p_i & \eqdef \hspace{1.24cm} |\set{j \in [2,n-1]}{j<i \text{ and } \pi_j<\pi_{j+1}}|,  \\
q_i & \eqdef n-1-|\set{j \in [2,n-1]}{j<i \text{ and } \pi_j>\pi_{j+1}}|.  
\end{align*}
In other words, the pairs of diagonals~$\posToDiagc(\pi_2), \dots, \posToDiagc(\pi_{n-1})$ form a centrally symmetric pair of accordions based on  the diagonals~$\posToDiagc(\pi_2) = \diag{0}{n-1} \cup \diag{\bar 0}{\overline{(n-1)}}$. We denote by~$Z_c$ the centrally symmetric pseudotriangulation formed by the diagonals~$\posToDiagc(\pi_0), \dots, \posToDiagc(\pi_{n-1})$. Note that the quiver~$\quiver(Z_c)$ is the Dynkin diagram of~$D_n$ oriented according to~$c$.
Finally, the other values of~$\posToDiagc$ are determined using the rotation map~$\tausqc$. Namely, $\posToDiagc(\tausqc(i))$ is obtained by rotating by~$\pi/n$ the pair of chords~$\posToDiagc(i)$, and exchanging~$\diagD{p}{L}$ with~$\diagD{p}{R}$. See Example~\ref{exm:bijection}.

\begin{example}
\label{exm:bijection}
We have computed this bijection for the Coxeter element $c = \tau_1\tau_2\tau_0$ of~$D_3$ in the following table. The columns corresponding to the pseudotriangulation~${\{\diagD{2}{L}, \diagD{\bar 2}{L}, \diagD{1}{L}, \diagD{\bar 1}{L}, \diag{1}{\bar 2}, \diag{\bar 1}{2}\}}$ are shaded. The reader is invited to check that the complement of these columns forms a reduced expression of~$\wo$ in~$\Qc$.
For this, it is useful to interpret the Coxeter group of type~$D_n$ in terms of signed permuations of~$[n]$ with and even number of negative entries. The generator~$\tau_0=(\bar 1\ \bar 2)$ is the transposition that interchanges 1 and 2 and changes their signs. The other generators~$\tau_i=(i\ i+1)~$ are the usual simple transpositions of the symmetric group. The longest element is the signed permutation~$\bar 1 \bar2 \bar 3 \dots \bar n$ if~$n$ is even, or~$1\bar 2\bar 3 \dots \bar n$ if~$n$ is odd. 

\bigskip
\renewcommand{\arraystretch}{1.6}
\centerline{\begin{tabular}{|x||g|x|x|x|x|x|g|g|x|}
	\hline
	Position $j$ in~$[9]$ & $1$ & $2$ & $3$ & $4$ & $5$ & $6$ & $7$ & $8$ & $9$ \\
	letter $q_j$ of~$\Qc$ & $\tau_1$ & $\tau_2$ & $\tau_0$ & $\tau_1$ & $\tau_2$ & $\tau_0$ & $\tau_1$ & $\tau_2$ & $\tau_1$ \\
	\twolines{3cm}{c.s.~pair of}{chords~$\posToDiagc(j)$ in~$\configD_3$} & $\diagD{2}{L} \cup \diagD{\bar 2}{L}$ & $\diag{0}{2} \cup \diag{\bar 0}{\bar 2}$ & $\diagD{0}{L} \cup \diagD{\bar 0}{L}$ & $\diagD{0}{R} \cup \diagD{\bar 0}{R}$ & $\diag{0}{\bar 1} \cup \diag{\bar 0}{1}$ & $\diagD{1}{R} \cup \diagD{\bar 1}{R}$ & $\diagD{1}{L} \cup \diagD{\bar 1}{L}$ & $\diag{1}{\bar 2} \cup \diag{\bar 1}{2}$ & $\diagD{2}{R} \cup \diagD{\bar 2}{R}$ \\
	\hline
\end{tabular}}
\bigskip
\end{example}

This example illustrates the main connection of this section. Namely, the bijection~$\posToDiagc$ relates centrally symmetric pseudotriangulations of~$\configD_n$ to reduced expressions in the word~$\Qc$ as follows.

\begin{proposition}
The complement of a set~$I \subset [m]$ of positions forms a reduced expression of~$\wo$ in~$\Qc$ if and only if the set~$\posToDiagc(I)$ of centrally symmetric pairs of chords forms a centrally symmetric pseudotriangulation of~$\configD_n$.
\end{proposition}

\begin{remark}
This statement rephrases the connection between cluster algebras and subword complexes, defined by A.~Knutzon and E.~Miller in~\cite{KnutsonMiller-subwordComplex}. For an element~$w$ and a word~$\Q$ in the generators of a finite Coxeter group, the \defn{subword complex}~$\subwordComplex(\Q,w)$ is the simplicial complex whose ground set is the set of positions in~$\Q$ and whose facets are the complements of reduced expressions of~$w$ in~$\Q$. For any finite Coxeter group~$W$ and any Coxeter element~$c$ of~$W$, C.~Ceballos, J.-P.~Labb\'e and C.~Stump proved in~\cite{CeballosLabbeStump} that the cluster complex of type~$W$ is isomorphic to the subword complex~$\subwordComplex(\Qc, \wo)$.
\end{remark}

\begin{remark}
The connection between type~$D$ cluster algebras and type~$D$ subword complexes can also be seen using the duality between pseudotriangulations and pseudoline arrangements in the M\"obius strip studied by V.~Pilaud and M.~Pocchiola in~\cite{PilaudPocchiola}. We briefly sketch this duality.

\begin{figure}[b]
	\centerline{
		\begin{overpic}[width=1.15\textwidth]{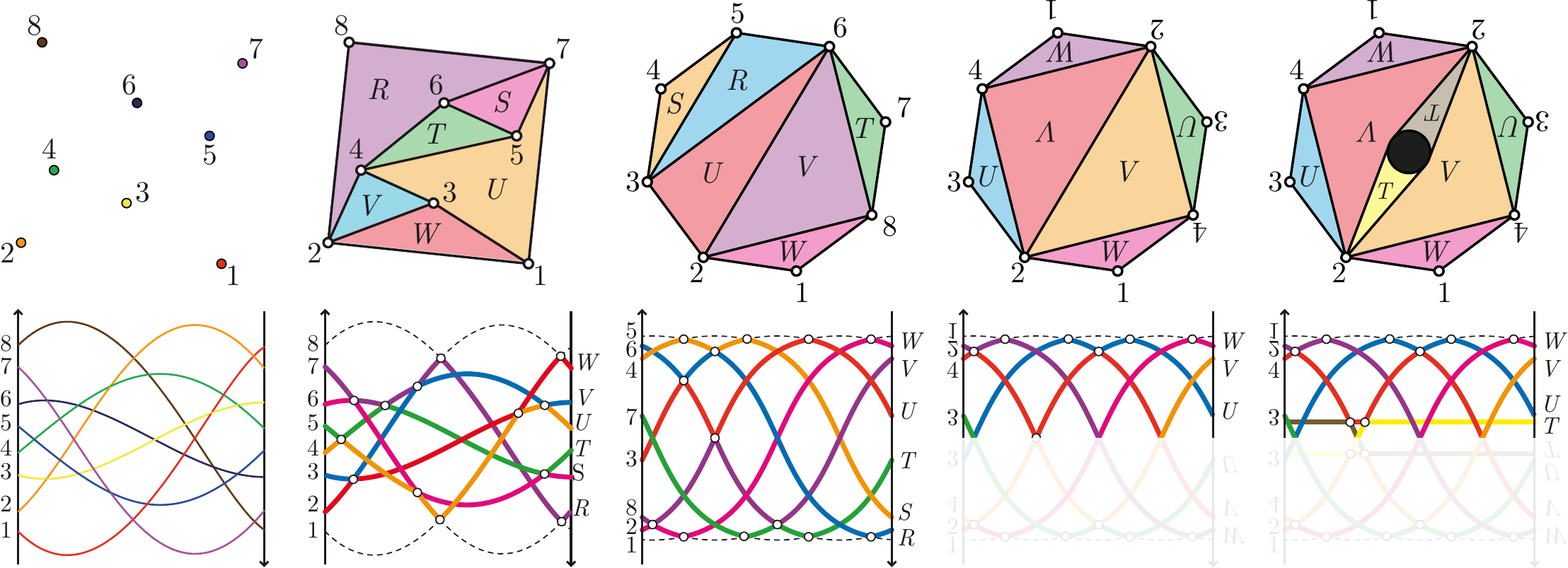}
			\put( 8  ,-2){(a)}
			\put(27.5,-2){(b)}
			\put(48  ,-2){(c)}
			\put(68.5,-2){(d)}
			\put(89  ,-2){(e)}
		\end{overpic}	
	}
	\vspace{.5cm}
	\caption{Duality between geometric configurations in~$\R^2$ (top) and pseudoline arrangements in the line space~$\cM$ (bottom)~\cite{PilaudPocchiola}.}
	\label{fig:duality}
\end{figure}
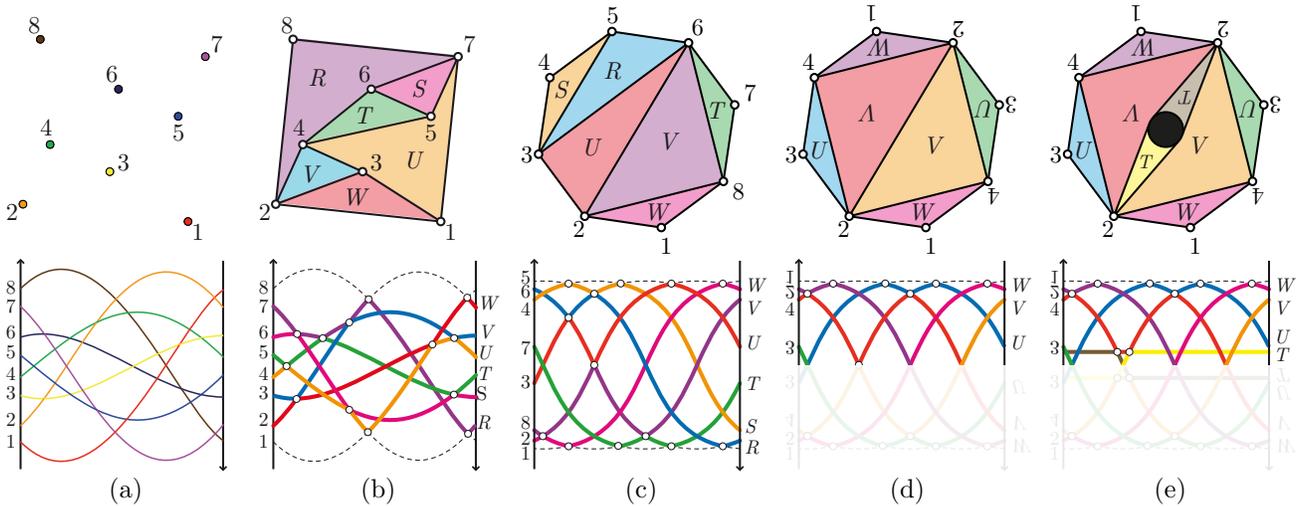

We parametrize an oriented line in~$\R^2$ by its angle~$\theta$ with the horizontal axis and its algebraic distance~$d$ to the origin. To forget the orientation, we identify the parameters~$(\theta,d)$ and~${(\theta + \pi, -d)}$. The unoriented line space of~$\R^2$ is thus the M\"obius strip~${\cM \eqdef \R^2/(\theta,d) \sim (\theta + \pi, -d)}$. A line~$\ell$ of~$\R^2$ yields a point~$\ell^*$ of~$\cM$, a point~$p$ of~$\R^2$ yields a \defn{pseudoline}~$p^* \eqdef \set{\ell^*}{p \in \ell}$ of~$\cM$ (\ie a non-separating simple closed curve in~$\cM$), and a point set~$\b{P}$ in general position in~$\R^2$ (no three on a line) yields a \defn{pseudoline arrangement}~$\b{P}^* \eqdef \set{p^*}{p \in \b{P}}$ of~$\cM$ (\ie a collection of pseudolines where any two cross precisely once). See \fref{fig:duality}\,(a). Note that lines joining two points of~$\b{P}$ correspond to the crossings in~$\b{P}^*$, and that the tangents to~$\conv(\b{P})$ correspond to the external level of~$\b{P}^*$ (touching the unbounded face of~$\b{P}^*$). 

Consider now a pseudotriangulation~$T$ of~$\b{P}$. Each pseudotriangle~$\pseudotriangle$ of~$T$ yields again a pseudoline~$\pseudotriangle^* \eqdef \set{\ell^*}{\ell \text{ internal tangent to } \pseudotriangle}$. The set~$T^* \eqdef \set{\pseudotriangle^*}{\pseudotriangle \in T}$ is a pseudoline arrangement and it precisely covers~$\b{P}^*$ minus its external level.  See \fref{fig:duality}\,(b). It is shown in~\cite{PilaudPocchiola} that this provides a bijective correspondence between the pseudotriangulations of~$\b{P}$ and the pseudoline arrangements which cover~$\b{P}^*$ minus its external level. These pseudoline arrangements can be seen as facets of a type~$A$ subword complex~$\subwordComplex(\Q,\wo)$ for some word~$\Q$. For example, when~$\b{P}$ is in convex position, it provides a correspondence between triangulations (type~$A$ clusters) and facets of the type~$A$ subword complex~$\subwordComplex(\Qc, \wo)$.  See \fref{fig:duality}\,(c).

Observe now that the central symmetry in~$\R^2$ (around the origin) translates to an horizontal symmetry in~$\cM$. It follows that a centrally symmetric pseudotriangulation~$T$ of a centrally symmetric point set~$\b{P}$ is dual to a horizontally symmetric pseudoline arrangement~$T^*$ on~$\b{P}^*$. Erasing the bottom half of~$T^*$ yields an arrangement which can be interpreted as a facet of a type~$B$ subword complex. For example, when~$\b{P}$ is centrally symmetric and in convex position, it provides a correspondence between centrally symmetric triangulations (type~$B$ clusters) and facets of the type~$B$ subword complex~$\subwordComplex(\Qc, \wo)$.  See \fref{fig:duality}\,(d).

Finally, consider the configuration~$\configD_n$ introduced in this paper. The disk~$D$ in~$\R^2$ yields a \defn{double pseudoline}~$D^* \eqdef \set{\ell^*}{\ell \text{ tangent to } D}$ of~$\cM$ (\ie a separating simple closed curve in~$\cM$). The dual~$\configD_n^*$ of the configuration~$\configD_n$ thus has $2n$ pseudolines and one double pseudoline. A centrally symmetric pseudotriangulation~$T$ of~$\configD_n$ is dual to a horizontally symmetric pseudoline arrangement~$T^*$ on~$\configD_n^*$. Erasing the bottom half of~$T^*$ yields an arrangement which can be interpreted as a facet of the type~$D$ subword complex~$\subwordComplex(\Qc, \wo)$.  See \fref{fig:duality}\,(e).
\end{remark}


\section{$c$-cluster complexes}

In this section, we provide a simple combinatorial description of $c$-cluster complexes of type~$D_n$ as described by Reading in~\cite{Reading-coxeterSortable}. These complexes are more general than the cluster complexes of Fomin and Zelevinsky~\cite{FominZelevinsky-YSystems}, and have an extra parameter $c$ corresponding to a Coxeter element. The particular case when $c$ is a bipartite Coxeter element recovers the cluster complexes of~\cite{FominZelevinsky-YSystems}. 
As in the previous section, consider the centrally symmetric accordion pseudotriangulation~$Z_c$ and centrally symmetric label its chords by $\{\pi_0,\dots, \pi_{n-1}\}$ corresponding to the letters $\tau_0,\dots ,\tau_{n-1}$ to which they correspond to in the Coxeter element $\sqc$. We identify the chords~$\{\pi_0,\dots, \pi_{n-1}\}$ with the negative simple roots~$\{-\alpha_0,\dots,-\alpha_{n-1}\}$, and any other chord $\delta$ with the positive root obtained by adding the simple roots associated to the chords of~$Z_c$ crossed by $\delta$. This gives a bijection between centrally symmetric pairs of chords and almost positive roots.   

We say that two almost positive roots $\alpha$ and $\beta$ are $c$-compatible if their corresponding pairs of chords do not cross. The $c$-cluster complex is the simplicial complex whose faces correspond to sets of almost positive roots that are pairwise $c$-compatible. The maximal simplices in it are called $c$-clusters and correspond naturally to centrally symmetric pseudotriangulations of~$\configD_n$. For instance, the accordion pseudotriangulation~$Z_c$ for $c=\tau_1\tau_2\tau_0$ of type~$D_3$ is illustrated in the middle of Figure~\ref{fig:typeD}. The $c$-cluster corresponding to the pseudotriangulation~$T$ in the right part of the same figure is $\{-\alpha_1,\alpha_2,\alpha_0+\alpha_2\}$. For example, the chord $[1,\bar 2]$ of $T$ corresponds to the positive root $\alpha_0+\alpha_2$ because it crosses the chords of $Z_c$ labeled by $\pi_0$ and $\pi_2$.  


\bibliographystyle{alpha}
\bibliography{CeballosPilaud_perfectMatchings}
\label{sec:biblio}


\end{document}